\documentclass[onefignum,onetabnum]{siamart190516}

\usepackage{amsmath,amsfonts,amssymb}
\usepackage{graphicx}
\usepackage{epstopdf}
\usepackage{algorithmic}
\usepackage{hyperref}
\usepackage{enumitem}
\usepackage{dsfont}
\ifpdf
  \DeclareGraphicsExtensions{.eps,.pdf,.png,.jpg}
\else
  \DeclareGraphicsExtensions{.eps}
\fi

\usepackage{tikz}
\usetikzlibrary{calc,patterns}
\tikzstyle{bag} = [align=center]

\usepackage{caption}
\usepackage{subcaption}
\usepackage{multirow}

\newcommand{\numSamples}{N}

\newcommand{\sampleSet}{X}

\newsiamremark{remark}{Remark}
\newsiamremark{hypothesis}{Hypothesis}
\crefname{hypothesis}{Hypothesis}{Hypotheses}
\newsiamthm{claim}{Claim}

\headers{Low Rank Saddle Free Newton}{T. O'Leary-Roseberry, N. Alger, O. Ghattas}

 \title{Low Rank Saddle Free Newton: A Scalable Method for Stochastic Nonconvex Optimization\thanks{This research was partially funded by the U.S. Department of Energy under ASCR awards DE-SC0019303 and DE-SC0021239 and ARPA-E award DE-AR0001208; the U.S. Department of Defense under MURI award FA9550-21-1-0084; the Simons Foundation under award 560651; and the National Science Foundation under award DMS-2012453.}}

\author{Thomas O'Leary-Roseberry\thanks{Oden Institute for Computational Engineering \& Sciences,
  The University of Texas at Austin, Austin, TX
  (\email{tom@oden.utexas.edu}, \email{nalger@oden.utexas.edu}).}
\and Nick Alger\footnotemark[2]
\and Omar Ghattas \thanks{Oden Institute for Computational Engineering \& Sciences,
  Department of Mechanical Engineering, and Department of Geological
  Sciences, The University of Texas at Austin, Austin, TX
  (\email{omar@oden.utexas.edu}).}}

\usepackage{amsopn}

\begin{document}

\maketitle

\begin{abstract}
In modern deep learning, highly subsampled stochastic approximation (SA) methods are preferred to sample average approximation (SAA) methods because of large data sets as well as generalization properties. Additionally, due to perceived costs of forming and factorizing Hessians, second order methods are not used for these problems. In this work we motivate the extension of Newton methods to the SA regime, and argue for the use of the scalable low rank saddle free Newton (LRSFN) method, which avoids forming the Hessian in favor of making a low rank approximation. Additionally, LRSFN can facilitate fast escape from indefinite regions leading to better optimization solutions. In the SA setting, iterative updates are dominated by stochastic noise, and stability of the method is key. We introduce a continuous time stability analysis framework, and use it to demonstrate that stochastic errors for Newton methods can be greatly amplified by ill-conditioned Hessians. The LRSFN method mitigates this stability issue via Levenberg-Marquardt damping. However, generally the analysis shows that second order methods with stochastic Hessian and gradient information may need to take small steps, unlike in deterministic problems. Numerical results show that LRSFN can escape indefinite regions that other methods have issues with; and even under restrictive step length conditions, LRSFN can outperform popular first order methods on large scale deep learning tasks in terms of generalizability for equivalent computational work.
\end{abstract}



\section{Introduction}
We are concerned with high dimensional stochastic nonconvex optimization problems of the form
\begin{align} 
    \min_{w\in \mathbb{R}^d} F(w) = \int \ell(w; x,y) d\nu(x,y) \label{exp_risk_min}.
\end{align}
The vector $w \in \mathbb{R}^d$ represents the optimization parameters, $\ell$ is a smooth loss function, and $(x,y) \in \mathbb{R}^{d_X}\times \mathbb{R}^{d_Y}$ are data pairs distributed with joint probability distribution $\nu(x,y)$. In practice $\nu$ is not known, but one has a sample dataset $\sampleSet = \{(x_i,y_i) \sim \nu\}_{i=1}^{\numSamples}$. Monte Carlo approximation of \eqref{exp_risk_min} yields:
\begin{equation} \label{emp_risk_min}
  \min_{w \in \mathbb{R}^d}F_X(w)= \frac{1}{\numSamples}\sum_{(x_i,y_i) \in \sampleSet} \ell(w;x_i,y_i).
\end{equation}
Ideally, searching for minima of \eqref{emp_risk_min} yields weights $w$ for which \eqref{exp_risk_min} is small---weights $w$ that \emph{generalize} to unseen data. Stochastic optimization problems are typically solved one of two ways: using all of the data $X$, which is referred to as \emph{sample average approximation} (SAA), and subsampling smaller data batches $X_k \subset X$, which is referred to as \emph{stochastic approximation} (SA) \cite{KimPasupathyHenderson2015}. The SAA approach treats the optimization problem as if it is deterministic.

Many tasks in modern computing involve the solution of \eqref{emp_risk_min}; in particular, challenges associated with deep learning have made SA solutions of \eqref{emp_risk_min} of notable interest in recent years. These problems are typically solved iteratively; starting from an initial guess, an update of the form:
\begin{equation} \label{iterative_update}
    w_{k+1} = w_k + \alpha_k p_k
\end{equation}
is applied until a convergence criterion is met. The vector $p_k$ is a search direction involving derivatives of the objective function, and $\alpha_k$ is the step length taken at that iteration. Canonical examples are gradient descent (GD), $p_k = -\nabla F(w_k)$, and full Newton $p_k = - [\nabla^2 F(w_k)]^{-1}\nabla F(w_k)$.

The solution of \eqref{emp_risk_min} via \eqref{iterative_update} has significant challenges due to the nonconvexity of $F$, the stochasticity of $\nu$, the need for hyperparameter tuning involved in the computation of $p_k$ and $\alpha_k$, and the computational costs involved in computing $p_k$ (due to the weight dimension, the dimensionality and cardinality of data $x,y$, as well as other computational costs). Modern applications rely on parallelizability / computational concurrency (e.g. GPU computing) for the efficient solution of these problems. With computational cost in mind, we investigate the extent to which high dimensional derivative information can be compressed to aid the solution of these stochastic optimization problems in the SA regime. In this work we argue for the use of the low rank saddle free Newton (LRSFN) algorithm \cite{OLearyRoseberry2020,OLearyRoseberryAlgerGhattas2019} in solving \eqref{emp_risk_min}. Our argument for LRSFN is based on its ability to facilitate fast escape from indefinite regions of parameter space and thus lead to better solutions and generalizability, as well as its ability to map well onto modern parallel computing architecutres. Since globalization procedures (e.g. line search) are not viable in the SA setting, this leads us to analyze the stability of the algorithm under stochastic perturbations to guide the choice of the step length parameter $\alpha_k$; we compare the stability of LRSFN to other canonical methods. We demonstrate the performance of the method relative to other standard methods on numerical examples ranging from nonconvex analytic functions to CIFAR\{10,100\} classification transfer learning using ResNet50. Our results indicate that LRSFN is well suited to handle indefiniteness in nonconvex problems, and lead to high quality solutions of SA optimization problems, so long as stability concerns are addressed via choice of appropriate step length.

\subsection{Background}

Second order methods are the methods of choice for deterministic problems and many SAA stochastic problems due to their superior convergence properties to first order methods. In the SA setting, where Monte Carlo subsampling errors dominate, classical quadratic convergence rates \cite{Boyd2004,NocedalWright2006} are unattainable due to constant error from gradient Monte Carlo terms. Superlinear convergence is possible for some problems with adaptive sampling strategies \cite{BollapragadaByrdNocedal2018}; this comes at significantly increased per-iteration costs when $d$ and $|X|$ are large. Stochastic gradient evaluation costs $O(N_{X_k}d)$, while formation and factorization of stochastic Hessian matrices formally requires $O(N_{X_k}d^2 + d^3)$ work, where $N_{X_k}$ is the number of data used in the Monte Carlo approximation. This ostensible cost disparity has led to the wide adoption of first order SA methods, where per-iteration costs are much lower than second order methods that form and factorize the full Hessian. These methods can be efficiently implemented in parallel computing environments, allowing for accelerated computation. 

Matrix free subsampled Newton methods \cite{BollapragadaByrdNocedal2018,ChenWuChenEtAl2019,OLearyRoseberryAlgerGhattas2019,RoostaMahoney2016a,RoostaMahoney2016b,YaoGholamiShenEtAl2020} require only \newline $O(N_{S_k}rd)$ work to form the Hessian. Here $r$ is the number of times the Hessian-vector product is formed, and $N_{S_k} \leq N_{X_k}$ is the size of the data subsampled for the Hessian. If $rN_{S_k}/N_{X_k} = O(1)$ the per-iteration costs of approximating Hessians, matrix-free, is comparable to that of first order methods. As is explained in Section \ref{section:compression}, additional work is required to invert the approximated Hessian, but formation of the Hessian-vector products is the dominant cost.

Nonconvexity is a difficult challenge for optimization methods. First order methods can escape saddle points asymptotically \cite{GeHuangJinEtAl2015,JinChiGeEtAl2017,LeePangeas2017,LeeSimchowitzJordan2016}, but indefiniteness can slow first order methods considerably. Indefiniteness is particularly problematic for Newton's method without modifications: applying an indefinite operator to a negative gradient does not guarantee a descent direction, and saddle points can be attractors for full Newton iterates. One can modify Newton's method to escape from saddle points quickly by replacing the Hessian, $H = \nabla^2 F_{S_k}$, with the absolute value of the Hessian, $|H|$. The matrix $|H|$ is a modification of $H$, where the negative eigenvalues are flipped to be positive \cite{GillMurrayWright1981}. That is,
\begin{equation*}
|H| := \sum_{j=1}^d |\lambda_j| u_j u_j^T,
\end{equation*}
where $\lambda_j$ are the eigenvalues of $H$ and $u_j$ are the corresponding eigenvectors. The eigenvalues are sorted such that $|\lambda_i|
\geq |\lambda_j|$ for all $i >j$. Flipping the sign of negative eigenvalues causes saddle points to repel iterates in directions of negative curvature, rather than attract them. Replacing $H$ with $|H|$ leads to the full Saddle Free Newton (full SFN) system
\begin{equation}
\label{sfn_system}
    |H| p = - g.
\end{equation}
Solving \eqref{sfn_system} for the saddle free Newton iterate, $p$, is infeasible in practice for several reasons. First, as discussed above, the formation and factorization of the full Hessian increases per-iteration costs much more than could be amortized in convergence benefits for high dimensional problems. Further, in applications like deep learning, the Hessian is itself often rank degenerate \cite{AlainRouxManzagol2019,GhorbanKrishnanXiaoi2019,OLearyRoseberry2020,OLearyRoseberryGhattas2020,SagunBottouLeCun2016}, so it cannot be directly inverted. In very high dimensions, the rank degeneracy of the Hessian actually becomes a benefit to be exploited, as we shall see. A low rank system can be inverted approximately using Levenberg-Marquardt damping,
\begin{equation}
\label{sfn_system_with_damping}
    (|H| + \gamma I) p = - g,
\end{equation}
where typically $1 \gg \gamma >0$ is the damping parameter. As was noted previously, the dominant information in the Hessian can be approximated efficiently via matrix-free methods for reasonable per-iteration costs ($O(N_{S_k} rd)$ work). The question becomes: which matrix-free method should be used to approximate the solution of \eqref{sfn_system_with_damping}?  

In the work of Dauphin et al.\ \cite{DauphinPescanuGulcehre2014} the Hessian is approximated in the subspace spanned by the $r$ Lanczos vectors of the Hessian for the Newton system. An eigenvalue decomposition of the Hessian in this $r$ dimensional subspace is then used to approximate $|H|$. This approximates the dominant Hessian subspace via an orthogonalization of a power iteration on one vector. The Lanczos method is inherently serial, since each matrix-vector product depends on the prior one, which makes this implementation unsuitable for modern parallel computing. Additionally, the Lanczos method has known numerical stability issues (e.g. sensitivity to starting vector, repeated eigenvalues, other issues) \cite{Cahill2000,Scott1979}. We consider instead directly approximating the low rank eigenvalue decomposition of $H$,
\begin{equation}
     H_r = \sum_{j=1}^r \lambda_j u_j u_j^T,
 \end{equation} 
via scalable parallelizable randomized algorithms \cite{HalkoMartinssonTropp2011,MartinssonTropp2020}. These methods can leverage computational concurrency via blocking (many matrix-vector products simultaneously), and have rigorous guarantees for the quality of the approximations of dominant eigenpairs for the matrix. Approximate solution of \eqref{sfn_system_with_damping} using a low rank approximation of $|H|$ leads to the LRSFN algorithm. As we will show, this method works well not just for low rank nonconvex problems, but also for the choice of $r$ that may be less than the numerical rank of the Hessian, as long as there is consistent decay in the dominant modes of the spectrum. What is left to figure out then, is how to choose an appropriate step length for the method in highly stochastic settings.

In deterministic or SAA stochastic optimization problems, globalization methods such as line search and trust region can be used at each iteration to select the step length $\alpha_k$ to enforce the \emph{sufficient descent} condition \cite{NocedalWright2006}. In SA optimization problems, statistical estimators of sufficient descent are unreliable due to Monte Carlo error and will likely lead to overfitting. This has led to the abandonment of globalization procedures in high dimensional and highly stochastic problems. The main issue then, is how does one choose $\alpha_k$? In the absence of statistical error, stability of optimizers depends on the optimization geometry.

For gradient based methods, bounds of the form $\alpha_k \leq \frac{C}{L}$ are common. Here $L$ is a gradient Lipschitz constant: $\|\nabla F(x) - \nabla F(y)\| \leq L\|x - y\|$ for all $x,y$ \cite{BottouCurtisNocedal2018,Boyd2004}. Full Newton's method is known to be asymptotically stable for $\alpha_k = 1$ \cite{Kelley1999}. When statistical error is introduced into the search direction, additional stability conditions can be derived in expectation with respect to the Monte Carlo error. When this error is large, one expects to take smaller steps; and when the error is small, one can take larger steps. We attempt to formalize bounds for LRSFN and understand how instabilities can propagate through the update procedures that take place in stochastic Newton methods, to help understand how to control them.

\subsection{Contributions}

In this work we argue for the LRSFN algorithm in the SA regime. Due to its ability to facilitate fast escape from indefinite regions, we expect it to converge to higher quality solutions---in particular better generalizability for stochastic problems. Since globalization procedures are not appropriate in this regime, the step length $\alpha_k$ must be chosen by other criteria. In order to motivate this choice, we derive a linearized stability theory for LRSFN and other optimizers (stochastic gradient descent (SGD) and stochastic Newton). In this theory we decompose stability issues into optimization geometry effects, and stochasticity effects. 

This analysis shows that while for first order methods, effects due to optimization geometry and stochasticity can be treated separately, for second order methods, stochastic errors can be greatly amplified due to effects of optimization geometry. Gradient errors can be amplified by Hessian ill-conditioning, and Hessian errors can be amplified by large gradients. So unlike in SAA or deterministic problems, where second order methods are able to take larger steps (i.e. $\alpha = 1$ for full Newton), or employ reliable globalization techniques, in the highly stochastic setting second order methods may require restrictive steps due to the potential for worst-case sensitivities to stochastic errors in highly oscillatory modes of the Hessian eigenbasis. We note that this inherent instability in full stochastic Newton is mitigated in LRSFN by the damping parameter, $\gamma$. 

In numerical results we illustrate the performance of LRSFN relative to other optimizers on a variety of deterministic and stochastic problems. In order to elucidate the stability theory, all experiments are carried out for fixed step lengths. These results illustrate how spectral structure, indefiniteness and stochasticity can affect the stability and convergence of these methods. They also demonstrate that compressing dominant Hessian information can have benefits that outweigh the resulting marginally increased per-iteration costs---even when restrictive step lengths are required. In particular LRSFN converges to significantly higher quality solutions than first order methods for two high dimensional classification transfer learning problems we investigated (CIFAR\{10,100\}). 

The rest of the paper is organized as follows: in Section \ref{section:compression} we argue for the use of randomized eigenvalue decompositions as an ideal matrix-free Hessian computational kernel for modern highly vectorized computing environments, in contrast to inherently serial matrix-free Hessian methods such as Krylov methods. In Section \ref{section:stability} we analyze the stability of stochastic optimizers to help guide the choice of the step length parameter $\alpha_k$ in LRSFN and other methods. In Section \ref{section:numerical_results} we evaluate LRSFN and other optimizers empirically on several test problems. 

The analysis in this paper is focused on stability of LRSFN; some analysis of the convergence properties of LRSFN can be found in \cite{OLearyRoseberryAlgerGhattas2019}.

\subsection{Notation} \label{section:notation}

By $\|\cdot\|$ we mean the $\ell^2$ norm (for matrices and vectors). When other norms such as Frobenius are mentioned, $\|\cdot\|_F$ will be used. By the set $X$, we mean the corpus of all training data, and use $X_k \subset X$ to denote the subsampled data used for the stochastic gradient at iteration $k$; likewise $S_k\subset X$ is the subsampled data used for the stochastic Hessian at iteration $k$. Set cardinalities are denoted $N_X,N_{X_k},N_{S_k}$; for deterministic problems involving a finite sum objective function (SAA problems), $N_{X_k} = N_{S_k} = N_X$. We denote by $\nabla F_{X_k}$ a subsampled gradient using batch $X_k$ at iteration $k$, and $\nabla^2F_{S_k}$ to be the subsampled Hessian using batch $S_k$ at iteration $k$. Expectations are subscripted to denote what measure they are with respect to. We denote by $\mathbb{E}_k$ the conditional expectation taken with respect to batches $X_k,S_k$ at iteration $k$. We use $\alpha_k$ and $\Delta t$ interchangeably for the step length parameter. We denote by $\mathds{1}^T\in\mathbb{R}^{d}$ the vector of all ones.

\section{Scalable compression of high dimensional Hessian information}\label{section:compression}

In this section we discuss the computational and numerical advantages of compressing Hessian information via randomized eigenvalue decompositions. As was noted in the previous section, the eigenvalue decomposition of the Hessian
\begin{equation}
    H = \sum_{j=1}^d \lambda_j u_ju_j^T = U\Lambda U^T,
\end{equation}
formally requires $O(N_{S_k}d^2)$ work to form and $O(d^3)$ work to factorize. For this reason Newton methods that require the entire Hessian are infeasible to use in high dimensional optimization problems. The Hessian can instead be approximated via its action on vectors $v \mapsto Hv$. For a fixed vector $v \in \mathbb{R}^{d}$, if the objective function $F$ is twice differentiable, we have
\begin{equation}
    Hv = \nabla^2F v = \nabla (\nabla F^Tv) = \nabla (g^Tv),
\end{equation}
that is, the Hessian-vector product on a vector $v$ is available by differentiating the inner-product of the gradient with $v$. This makes Hessian-vector products widely available in any setting where gradients are available via automatic differentiation (provided $F$ has sufficient regularity). This is known as the Pearlmutter trick \cite{Pearlmutter1994} in deep learning, where gradients are available via back-propagation (Lagrange multipliers). Many inexact Newton methods make use of Hessian-vector products to approximately invert the Hessian, such as in the Lanczos based SFN of Dauphin et al.\ \cite{DauphinPescanuGulcehre2014}, or other Krylov based inexact Newton methods such as inexact Newton CG \cite{BollapragadaByrdNocedal2018,OLearyRoseberryAlgerGhattas2019}. In all of these methods the approximate Newton inversion is implemented via an inherently sequential process of Hessian-vector products (each Hessian-vector products depends on the result of the previous Hessian-vector products).

In modern computing systems, latency is the dominant computational cost, which creates an advantage for computational concurrency, i.e. as much compute as possible in between reading data and communicating between threads. One way to do this is to employ the randomized eigenvalue decompositions proposed in \cite{HalkoMartinssonTropp2011,MartinssonTropp2020}. Instead of computing inherently serial Hessian-vector products, these methods can leverage de-serialized Hessian matrix products which are more appropriate for modern parallel computing. Given a tall and skinny matrix $V \in \mathbb{R}^{d\times r}$, one can compute a Hessian-matrix product in the same way the Hessian-vector product is approximated:
\begin{equation}
    HV = \nabla^2F V = \nabla (\nabla F^TV) = \nabla (g^TV).
\end{equation}
Since the action on each column of $V$ is independent, these methods can take advantage of vectorization by column partitioning $V$ into skinner sub-matrices. The Hessian-matrix products can be computed concurrently across many processes. Randomized methods, such as the double-pass algorithm sample the action of the Hessian on a matrix sampled from a random distribution $\rho$ in order to approximate the dominant range space of the matrix, i.e. they compute an orthonormal matrix $Q\in \mathbb{R}^{d \times r}$ such that 
\begin{equation}
    \mathbb{E}_\rho[\| H - HQQ^T\| ] < \epsilon,
\end{equation}
for some tolerance $\epsilon > 0$. Due to results in concentration of measure, randomized methods are robust in approximating the range of the Hessian \cite{HalkoMartinssonTropp2011}; this is a key feature of randomized methods since Krylov based subspace methods can have numerical stability issues \cite{Cahill2000,Scott1979}. This first step requires $O(N_{S_k}dr +dr^2)$ work (the additional term comes from QR factorization). Given the range approximation of the Hessian (the matrix $Q$), then one can form the Hessian in the column space of $Q$, i.e. the \emph{small} matrix $T = Q^THQ \in \mathbb{R}^{r \times r}$. Forming the small matrix $T$ takes again $O(N_{S_k}dr)$ work. The eigenvalue decomposition of $T = V_r\Lambda_rV_r^T$ then only requires $O(r^3)$ work which is negligible when $r \ll d$. A rank $r$ approximation of the eigenvalue decomposition of $H$ is then given by
\begin{equation}
    H_r = (QV_r)\Lambda_r(QV_r)^T = U_r \Lambda_r U_r^T.
\end{equation}
The accuracy of this approximate randomized low rank factorization is bounded in $\rho$-expectation by the trailing eigenvalues of the Hessian:
\begin{align}
    \mathbb{E}_\rho[\|H - U_r\Lambda_rU_r^T\|_{\ell^2(\mathbb{R}^{d\times d})}] &\leq C_1 |\lambda_{r+1}| \\ 
    \mathbb{E}_\rho[\|H - U_r\Lambda_rU_r^T\|_{F(\mathbb{R}^{d\times d})}] &\leq C_2 \sum_{j=r+1}^d|\lambda_j|,
\end{align}
where the constants $C_1,C_2$ do not depend on the Hessian but instead on hyperparameters for the randomized methods, and the distribution $\rho$ \cite{HalkoMartinssonTropp2011,MartinssonTropp2020}.

A major benefit of the randomized eigenvalue decomposition framework is that the rank of the matrix can be detected economically via an iterative procedure. In what is known as adaptive range finding / randomized QB factorization, one can continue to increase the dimension of the range until an error approximation is below a given tolerance. We note that adaptive range finding is not ideal for modern deep learning computational frameworks, where all instructions are typically compiled prior to execution. This is because it involves norm checking conditions to break a while loop, which can create cache misses if the routine exits sooner than the compiler expects. The adaptive range finding procedure can be used sparingly to detect the rank of the Hessian, and then fix it for future compiler optimized steps, recomputing occasionally, in order to minimize exposure to computational slow downs. We note the advantages of adaptive range finding but use fixed rank in numerical results.

The result of the randomized eigenvalue decomposition of the Hessian is an approximation of the dominant $r$ eigenpairs of the Hessian, with rigorous error bounds for the approximation in expectation. The cost of computing the randomized low rank approximation of the Hessian is $O(N_{S_k}dr +dr^2 + r^3)$, which is $O(N_{S_k}dr)$ when $r$ is small. If less data are used for the Hessian than the gradient (i.e. $N_{S_k} < N_{X_k}$) and $r$ is small, then the per-iteration cost of the approximation is commensurate to the cost of forming the gradient. The complexity of the method can be made only marginally more than the cost of forming the gradient by choosing $N_{S_K},N_{X_k}$ and $r$ accordingly.

\subsection{The LRSFN Algorithm}
In this section we review the LRSFN algorithm, introduced in \cite{OLearyRoseberryAlgerGhattas2019}. As was noted before, failing to take the absolute value of negative eigenvalues will flip gradient components initially oriented towards descent to ascent direction. In general the search direction $p_k$ would be a mix of ascent and descent components of the gradient, and this could fail to either minimize or maximize the function. Enforcing positive definiteness of any operator applied to the gradient maintains descent \cite{NocedalWright2006}. If $u_j$ is a direction of negative curvature, rescaling $g^Tu_j$ by $|\lambda_j|$ can help facilitate faster escape from indefinite regions. Rescaling in this fashion, one eigenvector at a time in the dominant $r$ dimensional subspace leads to the low rank SFN algorithm (LRSFN).

When the low rank approximation of the Hessian $H \approx U_r\Lambda_rU_r^T$ is computed, one can solve the Levenberg-Marquardt damped LRSFN system using the Sherman-Morrison-Woodbury formula, at the cost of only $O(r)$ matrix-vector products \cite{OLearyRoseberryAlgerGhattas2019}. At iteration $k$, the LRSFN update is as follows:
\begin{align}
\label{lrsfn_direction_formula}
    (U_r|\Lambda_r|U_r^T + \gamma I)p_k = -g_k \\ 
    p_k = - \bigg[\frac{1}{\gamma}I_d - \frac{1}{\gamma^2}U_r \bigg(|\Lambda_r|^{-1} + \frac{1}{\gamma}I_r\bigg)^{-1}U_r^T \bigg]g_k.
\end{align}

An inspection of equation \eqref{lrsfn_direction_formula} shows that the gradient direction is rescaled by the inverse of the damping parameter in the orthogonal complement to the column space of $U_r$. The rescaling by the saddle free Hessian curvature direction is only applied in the column space of $U_r$, and the curvature information is modified slightly by the damping parameter. This gives rise to the interpretation of LRSFN as a gradient descent method rescaled by the damping parameter $\gamma$ in the orthogonal complement of $U_r$, and a saddle free Newton method modified by $\gamma$ in the column space of $U_r$. LRSFN can be seen as a limiting algorithm between gradient descent and Newton. It attempts to capture the computational economy of gradient descent with the convergence benefits of Newton in a way that is computationally frugal. The LRSFN update is applied iteratively with a step length parameter $\alpha_k$ according to the formula $w_{k+1} = w_k + \alpha_k p_k$. The choice of the step length parameter $\alpha_k$ is the focus of the next section.

\section{Stability of LRSFN}\label{section:stability}

In this section we investigate the stability properties of LRSFN and compare it to that of gradient descent (GD) and full Newton, in order to illustrate stability issues that may arise for stochastic second order optimizers, and help guide the choice of $\alpha_k$ for LRSFN. Of particular interest is to what extent there is an interplay between stability issues due to loss landscape geometry and stochastic errors. We derive stability bounds for LRSFN, full Newton and GD that attempt to decouple potential instabilities due to stochastic error, and landscape geometry. For first order methods, these two effects can be easily decoupled, but as we will show there is an interplay between optimization geometry and stochastic errors for stochastic second order methods. 

In what follows we attempt to quantify the sensitivity of the LRSFN algorithm to the choice of rank, batch size, and damping parameter $\gamma$ and understand how these parameters effect the appropriate choice of step length $\alpha_k$. This analysis helps one understand what situations are appropriate to use the LRSFN algorithm, based on properties of the geometry of the loss landscape, and to what extent there is interplay between instabilities due to the Monte Carlo error and the conditioning of the Hessian (the geometry of the stochastic objective function). We begin by conceiving of optimization algorithms as discretizations of a minimizing flow of the form:
\begin{align}
\label{generic_minimizing_flow}
    \frac{dw}{dt} &= p(w)\\
                w(0) &= w_0.
\end{align}


The update $p(w)$ involves derivative information of the expected risk function $F$ evaluated at $w$. This minimizing flow, \eqref{generic_minimizing_flow} is simulated via numerical discretization (explicit Euler time integrator). Additionally, when the optimization problem is stochastic, the discretization involves a Monte Carlo approximation of the expectation with respect to the data distribution $\nu$. In such a flow one is looking for stable stationary points, i.e. $w^*\in \mathbb{R}^{d}$ such that $p(w^*) = 0$, or is suitably small in expectation. The stability of such points is dictated by the spectral properties of the Jacobian of the right hand side $\nabla p(w)$ \cite{Strogatz2018}, specifically the sign of the real components of the eigenvalues. When the real components of the eigenvalues are strictly negative, the critical point is stable. 

For our purposes, $p(w)$ is computed by application of a (hopefully positive-definite) matrix to the negative gradient of $F$. In this case a stationary point of the minimizing flow coincides with a first-order stationary point in optimization. Likewise, the stability of the stationary point can be related to conditions for second-order optimality. The canonical example of such a minimizing flow is gradient flow, $p(w) = -\nabla F(w)$. A stationary point of gradient flow is stable when the eigenvalues of $-\nabla^2 F(w)$ are negative, i.e. the Hessian is positive definite. This agrees exactly with the classical definition of strict second order stationary points in optimization. 

In typical two-step iterative optimization we are looking for stationary points via the explicit Euler update (interchanging $\alpha_k$ and $\Delta t$ for step length)
\begin{equation}
    w_{k+1} = w_k + \Delta t p(w_k);
\end{equation}
this approximation introduces $O(\Delta t)$ discretization error, but more importantly stability issues. explicit Euler (EE) is notoriously unstable, and in numerical differential equations implicit time integrators are preferred (i.e. implicit Euler (IE) $w_{k+1} = w_k + \Delta t p(w_{k+1})$ or Crank-Nicolson which is the average of EE and IE). Implicit time integrators are not useful in optimization since they increase the per-iteration complexity significantly by introducing an additional nonlinear Newton iteration to each step. At each iteration of the time-stepping approximation, error is introduced to the path, causing it to diverge from the minimizing flow (the solution of \eqref{generic_minimizing_flow}, which exists given suitable regularity of $p(w)$). The introduced errors can be due to stochastic noise, discretization errors, or other numerical artifacts. We focus on the effects of Monte Carlo errors for stochastic optimization problems, and call these errors $\eta$. These deviations are unavoidable in practice, but one can attempt to bound the deviations to not stray \emph{too far} from the minimizing path; this will be our guiding maxim.  The introduced errors propagate through the iterative procedure; this propagation is a discretization of the following ODE.
\begin{equation}
    \label{pertubed_minimizing_flow}
    \frac{d}{dt}(w + \eta) = p(w + \eta) + \underbrace{p_\text{MC}(w+\eta) - p(w+\eta)}_{\xi_{\text{MC}}(w + \eta)},
\end{equation}
where $p_\text{MC}$ is the Monte Carlo approximation of the search direction $p$. Taylor expanding the $p$ term about $w$ we have
\begin{equation}
    \label{taylor_of_p}
    p(w+ \eta) = p(w) + \nabla p(w)\eta + O(\eta^2).
\end{equation}
Since $\eta = 0$ at the initial condition, and we are interested in bounding small perturbations of $\eta$, we can neglect the quadratic terms in $\eta$. Subtracting \eqref{generic_minimizing_flow} from \eqref{pertubed_minimizing_flow} and employing the Taylor approximation we then get the linearized ODE system for the deviation:
\begin{align}
    \frac{d \eta}{dt} &= -\frac{dw}{dt} + p(w) + \nabla p(w) \eta + \xi_\text{MC}(w+\eta) \nonumber \\
                    &= \nabla p(w) \eta + \xi_\text{MC}(w+\eta).
\end{align}
In the case that the minimization problem is deterministic we have $\xi_\text{MC} = 0$. In stability analysis we are interested in deriving restrictions on the step length so that we can bound the deviation $\eta$ when the system is discretized via explicit Euler:
\begin{align}
    \label{stability_system}
    w_{k+1} &= w_k + \Delta t p(w_k) \\ 
    \eta_{k+1} &= \eta_k + \Delta t \nabla p(w_k)\eta_k + \Delta t \xi_\text{MC}(w_k + \eta_k).
\end{align}
This system expresses how deviations from the discretized minimizing flow propagate forward through the discretized subsampled minimizing flow. We say that the discretized flow is \emph{linearly stable} if we can maintain $\mathbb{E}_k[\|\eta_k\|] < 1$ for all $k$, where $\mathbb{E}_k$ denotes the conditional expectation taken with respect to the batch sizes chosen for gradient and Hessian Monte Carlo approximations at iteration $k$. 

\begin{theorem}{Generic Bound for Expected Linear Stability of Explicit Stochastic Optimizers}
\label{generic_stability_bound_theorem}

For any $\zeta \in (0,1)$ if the following conditions are met:
\begin{subequations}
\label{zeta_stability_bounds}
\begin{align}
    &\mathbb{E}_k[\|[I+  \Delta t \nabla p(w_k)]\eta_k\|] < \zeta  \label{zeta_stability_bounds1}\\
    & \Delta t < \frac{1 - \zeta}{\mathbb{E}_k[\|\xi_\text{MC}(w_k + \eta_k)\|]}\label{zeta_stability_bounds2},
\end{align}
\end{subequations}
then the iterative update at iteration $k+1$ is linearly stable.
\end{theorem}

\begin{proof}
The result follows directly from the triangle inequality:
\begin{align}
    \mathbb{E}_k[\|\eta_{k+1}\|] \leq \mathbb{E}_k[\|\eta_k + \Delta t \nabla p(w_k)\eta_k\|] + \Delta t \mathbb{E}_k[\|\xi_\text{MC}(w_k+\eta_k)\|] < \zeta + 1 - \zeta = 1
\end{align}
\end{proof}
If the optimization problem is deterministic, then we take $\zeta = 1$ and disregard the condition regarding the Monte Carlo error. The result is simple, but powerful for studying how landscape geometry \eqref{zeta_stability_bounds1} and stochasticity \eqref{zeta_stability_bounds2} effect stability of explicit optimization methods. The first term in \eqref{zeta_stability_bounds1}  can be bounded \emph{conservatively} using Cauchy-Schwarz:
\begin{equation}
    \mathbb{E}_k[\|[I+  \Delta t \nabla p(w_k)]\eta_k\|] \leq \|I + \Delta t \nabla p(w_k)\|_{\ell^2(\mathbb{R}^{d\times d})}\mathbb{E}_k[\|\eta_k\|].
\end{equation}
If $\mathbb{E}_k[\|\eta_k\|] <1$, then the following condition typical of linearized ODE stability analysis \cite{Butcher2006} maintains the expected stability:
\begin{equation}
    \label{conservative_stability_zeta}
    \|I + \Delta t \nabla p(w_k)\|_{\ell^2(\mathbb{R}^{d\times d})} < \zeta.
\end{equation} 
Note that if $\nabla p(w_k)$ has any positive eigenvalues then no $\Delta t > 0$ exist that satisfy the conservative bound \eqref{conservative_stability_zeta}. If $\nabla p(w_k)$ is negative semi-definite, then we can satisfy this conservative bound if 
\begin{equation}
    \label{typical_explicit_euler_stability_bound}
    \Delta t < \frac{1 + \zeta}{|\lambda_1(\nabla p(w_k))|}.
\end{equation}
The original perturbation growth condition in \eqref{zeta_stability_bounds1} can still be met if $\nabla p(w_k)$ has positive eigenvalues, but this depends on how large the coefficients of $\eta_k$ are in the corresponding eigenvectors, and this is hard to say much about in general. Coefficients of $\eta_k$ aligned with eigenvectors corresponding to positive eigenvalues will experience local exponential growth. Since the negative directions promote exponential decay, a noise coefficient amplified at one iteration could be later diminished if the Rayleigh-quotient for that direction changes sign between iterations.

In order to derive stability bounds for specific stochastic optimizers, we begin with a standard assumption used to bound the Monte Carlo error for the subsampled gradient (see for example \cite{BollapragadaByrdNocedal2018}).
\begin{enumerate}[label=A\arabic*]
    \item \label{assumption_grad_mc} (Bounded variance of sample gradients) There exists a constant $v$ such that
    \begin{equation}
    tr(\text{Cov}(\nabla F_i(w))) \leq v^2 \quad \forall w \in \mathbb{R}^d
    \end{equation}

\end{enumerate}

\begin{theorem}{Linear Stability of (Stochastic) Gradient Descent}
\label{theorem_stability_grad_descent}
For any $\zeta \in (0,1)$, if the Hessian $\nabla^2 F(w_k)$ is semi-positive definite, and assumption \ref{assumption_grad_mc} holds, then (stochastic) gradient descent
\begin{equation}
    w_{k+1} = w_k - \Delta t \nabla F_{X_k}(w_k),
\end{equation}
is linearly stable, so long as
\begin{equation}
 \Delta t < \min \left\{\frac{1 + \zeta}{|\lambda_1(\nabla^2 F(w_k))|}, \frac{(1 - \zeta)\sqrt{N_{X_k}}}{v} \right\}.
\end{equation}

\end{theorem}

\begin{proof}
In this construction we have $p(w_k) = - \nabla F(w_k)$ and \linebreak $\xi_\text{MC}(w_k) = - \nabla F_{X_k}(w_k) + \nabla F(w_k)$. We can bound
\begin{equation}
   \mathbb{E}_k[ \|\xi_\text{MC}(w)\|] \leq \frac{v}{\sqrt{N_{X_k}}},
\end{equation}
using Lemma \ref{grad_mc_lemma} in Appendix \ref{section:appendix_bounds}. The result follows from \eqref{typical_explicit_euler_stability_bound} and Theorem \ref{generic_stability_bound_theorem}.
\end{proof}
The stable step length must meet two requirements: one stemming from the most extreme local curvature of the expected risk landscape, and another involving the Monte Carlo error. These bounds are similar to existing bounds for GD \cite{Boyd2004}, as well as stochastic versions such as AdaGrad \cite{WardWuBottou2019}. In these bounds the role of the optimization geometry is encapsulated by the Lipschitz constant: $L$ such that $\|\nabla F(x) - \nabla F(y)\| \leq L\|x - y\| \forall x ,y$. Our bound instead uses the spectral norm of the local Hessian, which can be related to $L$ noting that via the mean value theorem and Cauchy-Schwarz; for twice differentiable $F$,
\begin{equation}
    \|\nabla F(x) - \nabla F(y)\| \leq \int_0^1 \|\nabla^2 F((1-t)x + ty)\|dt\|x-y\|.
\end{equation}



In order to proceed with analysis of the stability of stochastic second order optimizers, we must first assume that the stochastic Hessians have bounded variance.
\begin{enumerate}[label=A\arabic*]
    \setcounter{enumi}{1}

    \item \label{assumption_hess_mc} (Bounded variance of Hessian and absolute Hessian components) There exist $\sigma$, $\sigma^{\text{abs}}$ such that, for all component Hessians, we have
    \begin{subequations}
    \begin{align}
    \|\mathbb{E}[(\nabla^2 F_i(w) - \nabla^2F(w))^2]\| &\leq \sigma^2, \quad \forall w \in \mathbb{R}^d,\label{bounded_var_hess_comp} \\
    \|\mathbb{E}[(|\nabla^2 F_i(w)| - |\nabla^2F(w)|)^2]\| &\leq (\sigma^\text{abs})^2, \quad \forall w \in \mathbb{R}^d\label{bounded_var_abshess_comp}.
    \end{align}
    \end{subequations}
\end{enumerate}


This assumption leads to bounds for the Monte Carlo error of the search directions of both stochastic Newton, as well as (stochastic) LRSFN. While the effects of the curvature information and the stochastic update error can be fully decoupled for stochastic gradient descent, the next proposition demonstrates that there is interplay between the Hessian curvature information and the stochastic Monte Carlo errors for stochastic Newton methods.

\begin{proposition}{Monte Carlo Error for (Stochastic) Newton and LRSFN \linebreak search directions}
\label{proposition_newton_lrsfn_search}

Assume \ref{assumption_grad_mc} and \ref{assumption_hess_mc}. Defining $E_\text{MC}^\text{N} = \nabla^2 F_{S_k} - \nabla^2 F$, the Monte Carlo errors for the search directions in stochastic Newton has the following conservative bounds:
\begin{subequations}
\begin{align}
&\mathbb{E}_k[\|\nabla^2F_{S_k}^{-1}\nabla F_{X_k} - \nabla^2 F^{-1}\nabla F\|]\\
 &\leq \frac{C_0}{\sqrt{N_{X_k}}} + \frac{C_1}{\sqrt{N_{S_k}}}\mathbb{E}_k[\|\nabla F_{X_k}\|] +\frac{C_2}{N_{S_k}}\mathbb{E}_k[\|\nabla F_{X_k}\|]\\
&C_0 = v\|\nabla^2F^{-1}\label{newton_mc_c0}\| \\
&C_1 = \frac{\sigma}{2}\|\nabla^2F^{-1}\|^2  (1 + \mathbb{E}_k[\|(I + E_\text{MC}^\text{N}\nabla^2F^{-1})^{-1}]) \label{newton_mc_c1}\\
&C_2 = \frac{\sigma^2}{4}\|\nabla^2F^{-1}\|^3  \mathbb{E}_k\left[\left\|\left(I +\frac{1}{2}E^\text{N}_\text{MC}\nabla^2F^{-1}\right)^{-1}\right\|\right]\label{newton_mc_c2}.
\end{align}
\end{subequations}
Likewise defining $E_{MC}^\text{L} = |\nabla^2F_{S_k}| - |\nabla^2 F|$,for (stochastic) low rank saddle free Newton the bounds are
\begin{subequations}
\begin{align}
&\mathbb{E}_k[\|[|\nabla^2F^{(r)}_{S_k}|+\gamma I]^{-1}\nabla F^{(r)}_{X_k} - [|\nabla^2 F|+\gamma I]^{-1}\nabla F\|] \leq \nonumber \\ 
&\qquad\frac{C_0}{\sqrt{N_{X_k}}} + \frac{C_1}{\sqrt{N_{S_k}}}\mathbb{E}_k[\|\nabla F_{X_k}\|] +\frac{C_2}{N_{S_k}}\mathbb{E}_k[\|\nabla F_{X_k}\|]\\
&C_0 = \frac{v}{\gamma}\label{lrsfn_mc_c0}\\
&C_1 = \frac{\sigma^\text{abs}}{2\gamma^2}  (1 + \mathbb{E}_k[\|(I + E_\text{MC}^\text{N}[|\nabla^2F^{(r)}|+\gamma I]^{-1})^{-1}])\label{lrsfn_mc_c1} \\
&C_2 = \frac{(\sigma^\text{abs})^2}{4\gamma^3}  \mathbb{E}_k\left[\left\|\left(I +\frac{1}{2}E^\text{N}_\text{MC}[|\nabla^2F^{(r)}|+\gamma I]^{-1}\right)^{-1}\right\|\right]\label{lrsfn_mc_c2}.
\end{align}
\end{subequations}
\end{proposition}

\begin{proof}
This is a corollary of Lemma \ref{stochastic_newton_mc_lemma} in Appendix \ref{section:appendix_bounds}.
\end{proof}

These bounds are conservative, and in the case of full Newton, if the expected risk Hessian is rank degenerate, these conservative bounds will not exist (i.e. the worst case is uncontrollable instability). The key takeaway from these bounds is that gradient Monte Carlo errors can be greatly amplified by an ill-conditioned Hessian matrix, likewise, Hessian Monte Carlo errors can be greatly amplified by a large gradient. The constants $C_0$ represent how Monte Carlo errors in the gradient can be potentially amplified by inverting in directions corresponding to the smallest eigenvalues of the operator being applied to the gradient. For stochastic Newton, the bound \eqref{newton_mc_c0} suggests that the smallest eigenvalues of the Hessian can greatly amplify errors aligned with the corresponding eigenvectors, posing a significant stability issue. Equation \eqref{lrsfn_mc_c0} suggests that this is not nearly as bad of a problem for LRSFN, since the damping parameter $\gamma$ controls the amplifications, and one would not want to take $\gamma$ too small, since the Sherman-Morrison-Woodbury formula can itself run into stability issues. 

The constants $C_1$ and $C_2$ represent terms associated with the Hessian Monte Carlo error. For either of these constants to be less than infinity the expectations of the norms of matrices of the form $\mathbb{E}_k[\|(I+E_\text{MC}A^{-1})^{-1}\|],\mathbb{E}_k[\|(I+\frac{1}{2}E_\text{MC}A^{-1})^{-1}\|]$ must be bounded. Here $A$ is a generic representation of the operator inverted in the Newton update, and $E_\text{MC}$ represents the matrix Monte Carlo error. An obvious issue occurs when the spectra of the operators $E_\text{MC}A^{-1},\frac{1}{2}E_\text{MC}A^{-1}$ contain $-1$ exactly in expectation: a rare event. We can reasonably assume that this event will not occur asymptotically, since the spectrum of the operator can be bounded in expectation away from $-1$ using Cauchy-Schwarz:
\begin{equation}
    \mathbb{E}_k[\|E_\text{MC}A^{-1}\|] \leq \frac{\sigma_A}{\sqrt{N_{S_k}}}\|A^{-1}\| < 1,
\end{equation}
where $\sigma_A$ is the assumed bound on the variance of the stochastic matrix $A_{S_k}$. This condition is met in expectation when $N_{S_k} > \sigma_A^2\|A^{-1}\|^2$. The terms involving constants $C_1$ and $C_2$ both involve the norm of the gradient, so as iterates gets closer to a local minimizer, the effects of Hessian Monte Carlo error are diminished proportionally to the decay of the gradient norm in expectation. As with $C_0$, the constants $C_1,C_2$ involve the spectral norm of the operator being inverted in the update.

In addition to the stochastic Hessian error, there is the term involving the geometry of the optimization landscape / the sensitivity of the update direction with changes in the parameter, $\nabla p(w)$. For Newton, $p(w) = -[\nabla^2 F(w)]^{-1}\nabla F(w)$, and we can derive $\nabla p(w)$ via the product rule:
\begin{equation}
    \nabla p(w) = - \left(\nabla[\nabla^2 F(w)]^{-1} \right)\nabla F(w) - [\nabla^2F(w)]^{-1}\nabla^2 F(w).
\end{equation}
Assuming that the Hessian is invertible, we can employ the inverse matrix derivative identity:
\begin{equation}
    \nabla(A^{-1}) = - A^{-1}(\nabla A) A^{-1},
\end{equation}
and the formula is then
\begin{equation}
    \nabla p(w) = [\nabla^2 F(w)]^{-1} \nabla^3 F(w) [\nabla^2 F(w)]^{-1} \nabla F(w) - I.
\end{equation}
For low rank saddle free Newton, $p(w) = -(|\nabla^2 F^{(r)}| + \gamma I)^{-1}\nabla F$, and the search direction Jacobian is
\begin{equation}
    \nabla p(w) = (|\nabla^2 F^{(r)}| + \gamma I)^{-1} (\nabla |\nabla^2 F^{(r)}|)(|\nabla^2 F^{(r)}| + \gamma I)^{-1}\nabla F - I.
\end{equation}
In both the cases of Newton and LRSFN the update meets the stability condition stemming from optimization geometry if \eqref{zeta_stability_bounds1} is met for the given choice of $\zeta$. This condition is hard to verify, because verification requires knowledge of how much of the perturbation $\eta$ is aligned with eigenvectors of $\nabla p(w)$. The condition can be met \emph{conservatively} if \eqref{conservative_stability_zeta} is met, this requires that the additional condition that $\nabla p(w)$ is negative semi-definite. For Newton and LRSFN, this condition is met conservatively if the following spectral radius bounds are met:
\begin{subequations}
\begin{align}
  &\|[\nabla^2 F(w)]^{-1} \nabla^3 F(w) [\nabla^2 F(w)]^{-1} \nabla F(w)\| \leq 1 \\
  &\|(|\nabla^2 F^{(r)}| + \gamma I)^{-1} (\nabla |\nabla^2 F^{(r)}|)(|\nabla^2 F^{(r)}| + \gamma I)^{-1}\nabla F\| \leq 1.
\end{align}
\end{subequations}
What do these bounds mean? First order methods require bounds on second derivatives of the objective function for stability, and likewise second order methods require bounds on third derivatives. Both of these matrices are the product of the operator that is inverted acting on the directional derivative of the same operator in the search direction. This directional derivative is a contraction of a third derivative tensor on the search direction; if the Hessian operator is not changing too much along the path of the optimizer then this term can be reasonably bounded. For example, in quadratic optimization problems this term is identically zero. This is also a reasonable assumption in practice (at least sufficiently far from initial guesses), in the algorithm AdaHessian \cite{YaoGholamiShenEtAl2020}, the Hessian is only computed every few iterations and Hessian-vector products are reused; the algorithm is able to beat first order methods in difficult deep learning tasks. The inverse operator acting on this directional derivative matrix provides further issues perhaps, since a conservative bound via Cauchy-Schwarz will drastically increase the spectral radius. This is much more of an issue again for stochastic Newton than it is for LRSFN, since stochastic Hessians have many eigenvalues clustering near zero. When the iterates are in the viscinity of a local minimizer, the gradient norm approaches zero, making the condition $\nabla p(w) \preceq 0$ easier to meet, since the non-identity term is scaled by the gradient norm. Thus this stability analysis retrieves the classical stability of Newton methods in the local convergence regime. We have the following bounds for the stability of the second order stochastic Newton methods.

\begin{theorem}{Linear Stability of (Stochastic) Newton and Low Rank Saddle Free Newton}
\label{theorem_linear_stability_newton}
For any $\zeta \in (0,1)$, if the update Jacobian $\nabla p(w_k)$ is negative definite and bounded, and there exist constants $C_0,C_1,C_2$ as in Proposition \ref{proposition_newton_lrsfn_search} to bound the Monte Carlo error of the stochastic search direction, then (stochastic) Newton and LRSFN defined generically as
\begin{align}
     w_{k+1} &= w_k +\Delta t p(w_k) \\
     p(w_k) &= [\nabla^2F(w_k)]^{-1}\nabla F(w_k)  \qquad \text{(Newton)} \\
     p(w_k) &= [|\nabla^2F^{(r)}(w_k)| +\gamma I]^{-1}\nabla F(w_k)  \qquad \text{(LRSFN)}
\end{align} 
are linearly stable, so long as
\begin{equation}
 \Delta t < \min \left\{\frac{1 + \zeta}{|\lambda_1(\nabla p(w_k))|}, \frac{(1 - \zeta)}{\frac{C_0}{\sqrt{N_{X_k}}} + \frac{C_1}{\sqrt{N_{S_k}}}\mathbb{E}_k[\|\nabla F_{X_k}\|] + \frac{C_2}{N_{S_k}}\mathbb{E}_k[\|\nabla F_{X_k}\|]} \right\}.
\end{equation}
\end{theorem} 

\begin{proof}
This matrix $\nabla p(w)$ is negative semi-definite by assumption. Since the update Jacobian is bounded by assumption, the result follows from \eqref{typical_explicit_euler_stability_bound} and Theorem \ref{generic_stability_bound_theorem}.
\end{proof}

The takeaway for this result is that the assumptions to conservatively bound the stability of stochastic Newton are very restrictive. Instabilities in the search direction update for stochastic Newton can be drastically amplified by ill-conditioning in the stochastic Hessian. Beyond the obvious per-iteration cost issues in high dimensions, stochastic full Newton is unviable additionally because it is an inherently unstable algorithm. These inherent instabilities due to ill-conditioning are mitigated in LRSFN where the sensitivity to the inverse Hessian is replaced instead by the Levenberg-Marquardt damping parameter. Another important takeaway from this analysis is that while the sensitivity to the stochastic errors and landscape geometry are thoroughly decoupled for gradient descent (Theorem \ref{theorem_stability_grad_descent}), the two are coupled for stochastic second order methods (Proposition \ref{proposition_newton_lrsfn_search}). In these methods, errors in the stochastic gradient can be amplified by the curvature information of the approximated inverse Hessian, and likewise errors in the stochastic Hessian can be amplified by a large gradient. While second order methods are known to be able to take larger steps in deterministic and SAA problems, they may have to take smaller steps in the SA regime due to the interplay of ill-conditioning and statistical error.

The bounds we derived for gradient descent and Newton methods require weak convexity assumptions, which is a bad assumption in practical applications like deep learning. For stochastic second order methods this curvature condition instead requires bounds on the spectral norm of a term that is basically a Hessian-preconditioned directional derivative of the third derivative tensor ( $\nabla^3 F(w)$). This bound can hold even for strongly nonconvex problems, so formally stochastic Newton methods do not require as strict of convexity assumptions as gradient descent for conservative bounds on stability. Granted, if the Hessian-like operator that is inverted is not enforced to be positive-definite the resulting direction is not even guaranteed to be a descent direction. This is handled naturally by LRSFN, where the inverted operator is guaranteed to be positive-definite. In the presence of perfect gradient and Hessian information, the stability depends only on the optimization geometry. For convex problems, when the full Hessian is inverted and the problem has a well bounded third derivative tensor, one can take larger steps without creating instabilities, agreeing with classical Newton asymptotics.

In the next section we investigate the stability of LRSFN and other optimizers on deterministic and stochastic problems. The results demonstrate that even when very small steps are required in the presence of large stochastic noise and steep curvature information, LRSFN can still have significant upside over other conventional methods.

\section{Numerical Results} \label{section:numerical_results}

In this section we evaluate the numerical performance of the LRSFN algorithm relative to other methods on problems ranging from nonconvex analytic functions to highly stochastic deep neural network training. Since we are ultimately interested in the performance of LRSFN in the SA regime, we analyze the performance of the optimizers for fixed step lengths which are motivated by the analysis in Section \ref{section:stability}. LRSFN can (and should) be adapted to SAA/deterministic problems by employing globalization and adaptive range finding strategies, but that is outside of the scope of this work (see \cite{OLearyRoseberryAlgerGhattas2019}). We thus use deterministic problems as a means of understanding how effects from optimization geometry, absent of statistical noise, can effect the choice of step size for LRSFN, motivating step length considerations in the much more complicated SA regime. We do not consider additional hyperparameters such as reduction on plateau etc. for the sake of simplicity: any hyperparameter tuning that would benefit first order methods is also likely to benefit second order methods.

In the first examples, we consider nonconvex analytic functions with both quick spectral collapse (Michalewicz) and heavy tails (Rosenbrock), the latter demonstrates that LRSFN can have issues for problems with heavy tails the full SFN and Newton do not. The last examples are very high dimensional and stochastic problems from image classification transfer learning. In these examples we compare against standard methods which use very small sample sizes. As predicted by the stability analysis in Section \ref{section:stability}, subsampling errors seem to lead to higher instabilities for LRSFN than they do for Adam and SGD. However even when restrictive step lengths are used LRSFN can still outperform these first order methods for the same amount of compute.

\subsection{Analytic Functions}

In the first set of examples, we consider the \newline Michalewicz \cite{MolgaSmutnicki2005} and the d-dimensional Rosenbrock functions \cite{PaganiWiegand2019}, which can both be extended to arbitrary natural number dimension. All code can be found here \cite{matrixfreenewton}. For these problems we compare the performance of LRSFN with gradient descent, curvature scaled gradient descent (CSGD,where we take $\alpha = \frac{1}{\lambda_1(\nabla^2F)}$), full SFN and Newton. The Michalewicz function,
\begin{equation}
  F(w) = -\sum_{j=1}^d \sin(w_j)\sin^{20}\left(\frac{jw_j^2}{\pi}\right),
\end{equation}
has $d!$ local minima in the hypercube $[0,\pi]^d$, with very steep valleys, highly indefinite Hessian with quick rank decay, it is described as ``needle in haystack'' problem \cite{MolgaSmutnicki2005}. We sample ten different initial guesses around $0\in \mathbb{R}^{d}$, with $d = 100$, and show the results in a table below. For all optimizers we use $\Delta t = \alpha = 1$, and perform 100 optimization iterations. Since these low dimensional analytic functions are computationally negligible to evaluate we do not consider computational cost comparison between first and second order methods.

\begin{table}[H] 
\center
\begin{tabular}{|l|c|c|}
\hline
\enskip & Average minimum & Minimum minimum \\
\hline
GD & $-2.79 \pm 1.66$ & $-7.67$ \\ 
\hline
CSGD & $-8.21 \pm 2.09$ & $-12.41$ \\ 
\hline
Newton & $-0.29 \pm 1.39$ & $-2.28$ \\ 
\hline
Full SFN & $\mathbf{-24.26 \pm 2.65}$ & $\mathbf{-28.51}$ \\ 
\hline
LRSFN $r = 20$ & $-4.35 \pm 0.83$ &$ -5.88$ \\ 
\hline
LRSFN $r = 40$ & $-10.07 \pm 4.73$ &$ -17.78$ \\ 
\hline
LRSFN $r = 50$ & $-19.37 \pm 3.21$ &$ -24.48$ \\ 
\hline
LRSFN $r = 60$ & $-19.34 \pm 2.91$ &$ -24.71$ \\ 
\hline
LRSFN $r = 80$ & $-22.13 \pm 2.59$ &$ -26.35$ \\ 
\hline
\end{tabular}
\caption{Summary of results for Michalewicz optimization}
\label{table_michalewicz}
\end{table}

\begin{figure}[H] 
\begin{subfigure}{0.5\textwidth}
\includegraphics[width = \textwidth]{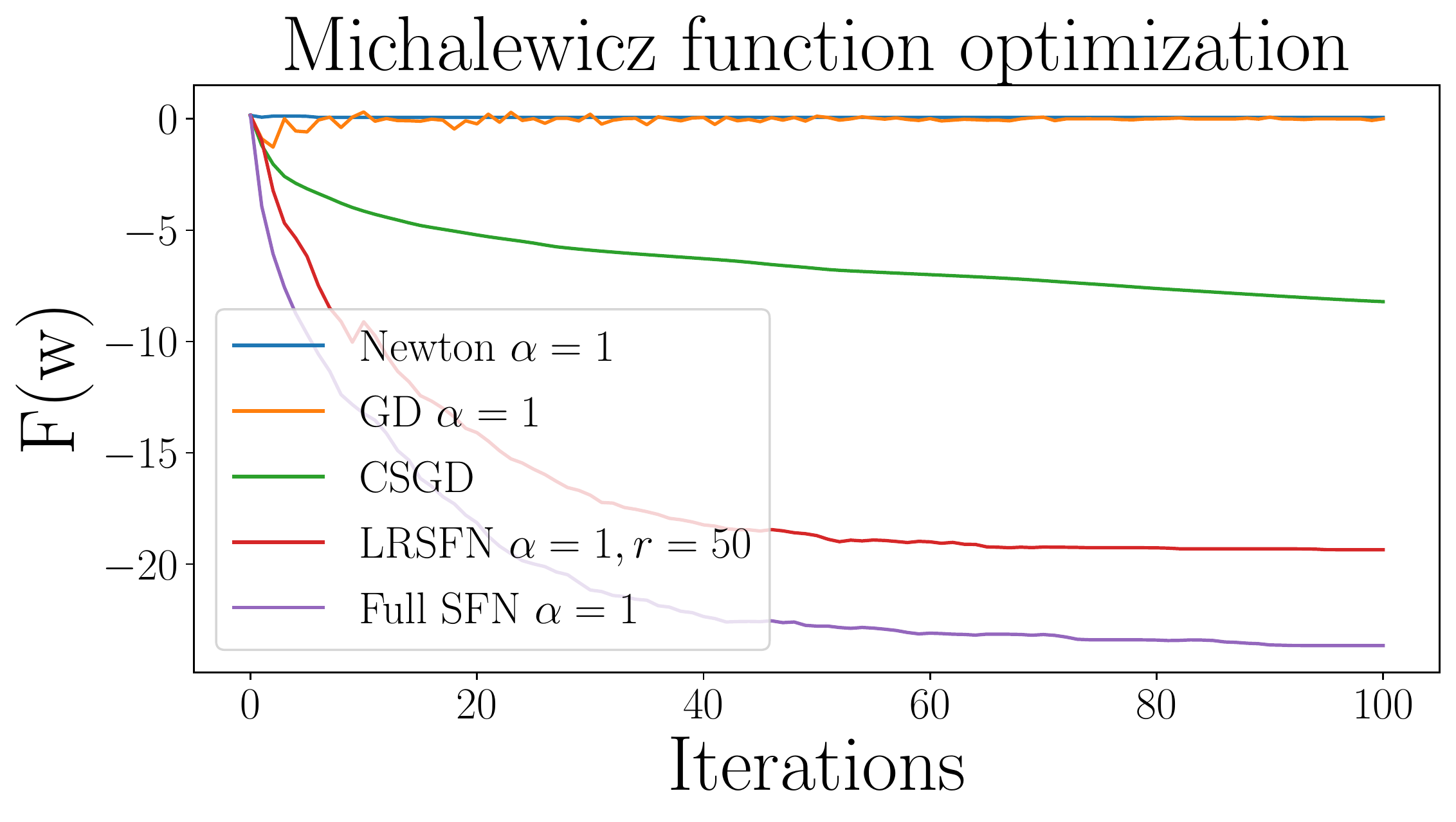}

\end{subfigure}%
\begin{subfigure}{0.5\textwidth}
\includegraphics[width = \textwidth]{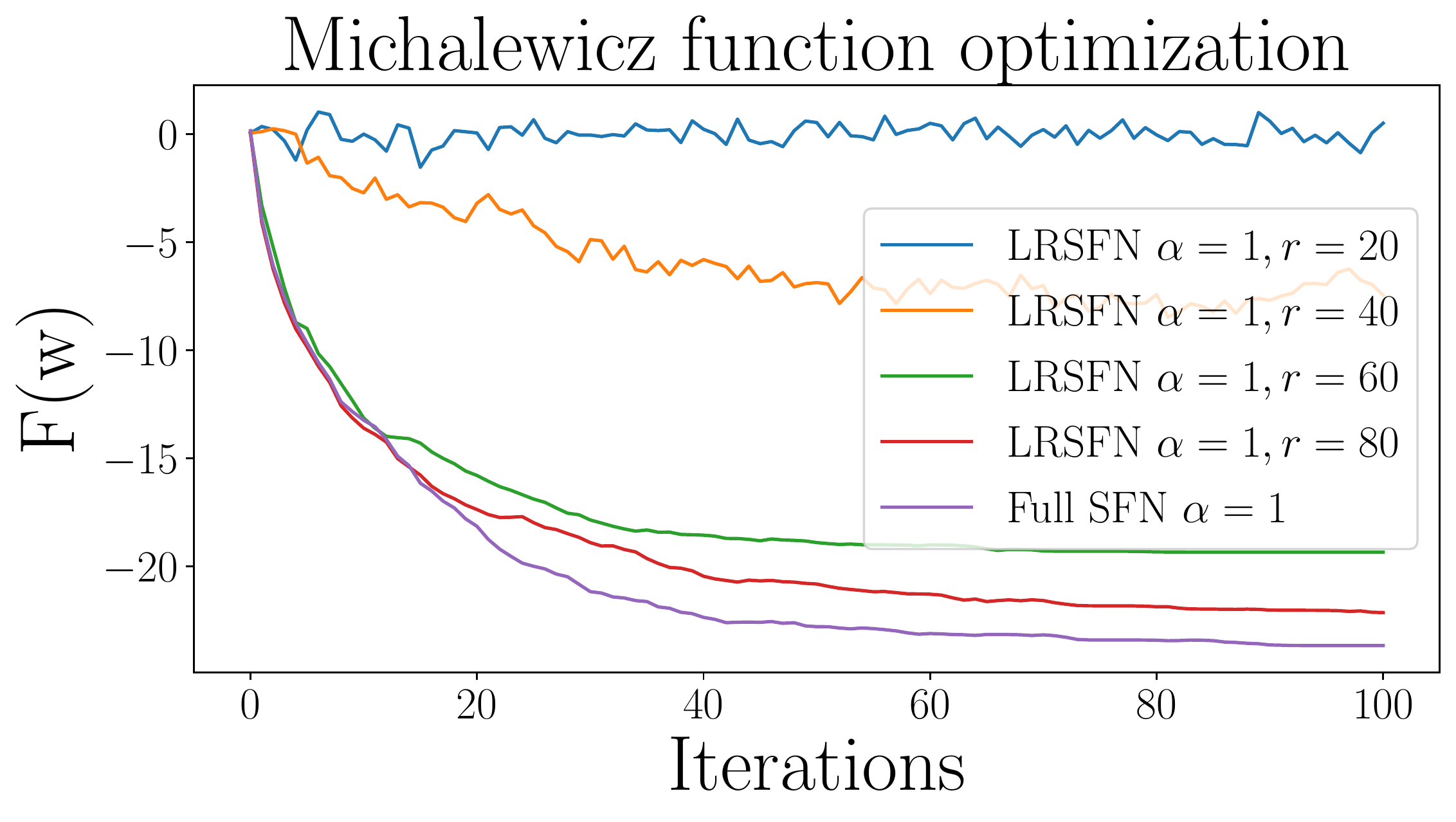}
\end{subfigure}
\caption{Michalewicz losses vs iterations averaged across all ten samples}
\label{michalewicz_losses}
\end{figure}

\begin{figure}[H] 
\begin{subfigure}{0.33\textwidth}
\includegraphics[width = \textwidth]{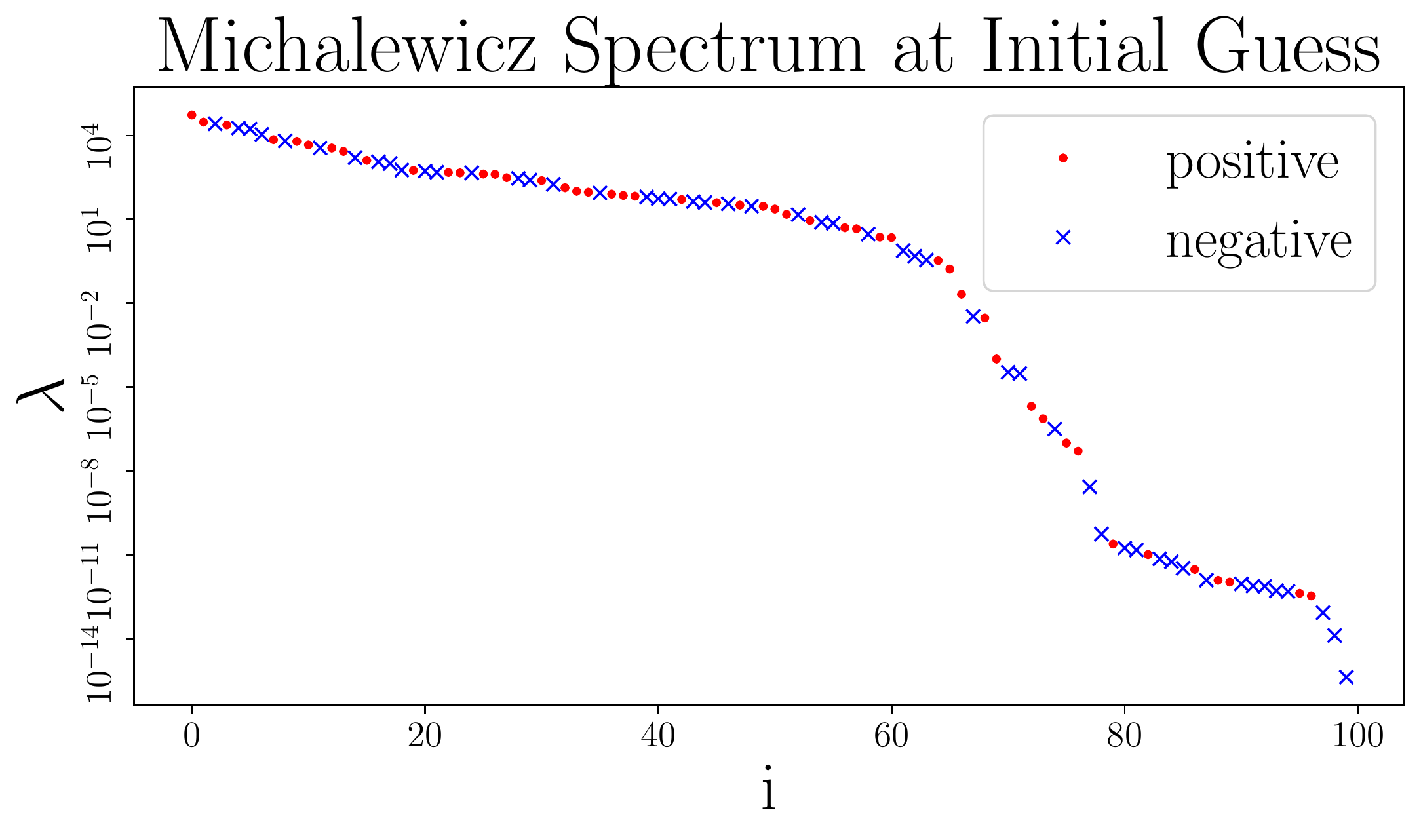}

\end{subfigure}%
\begin{subfigure}{0.33\textwidth}
\includegraphics[width = \textwidth]{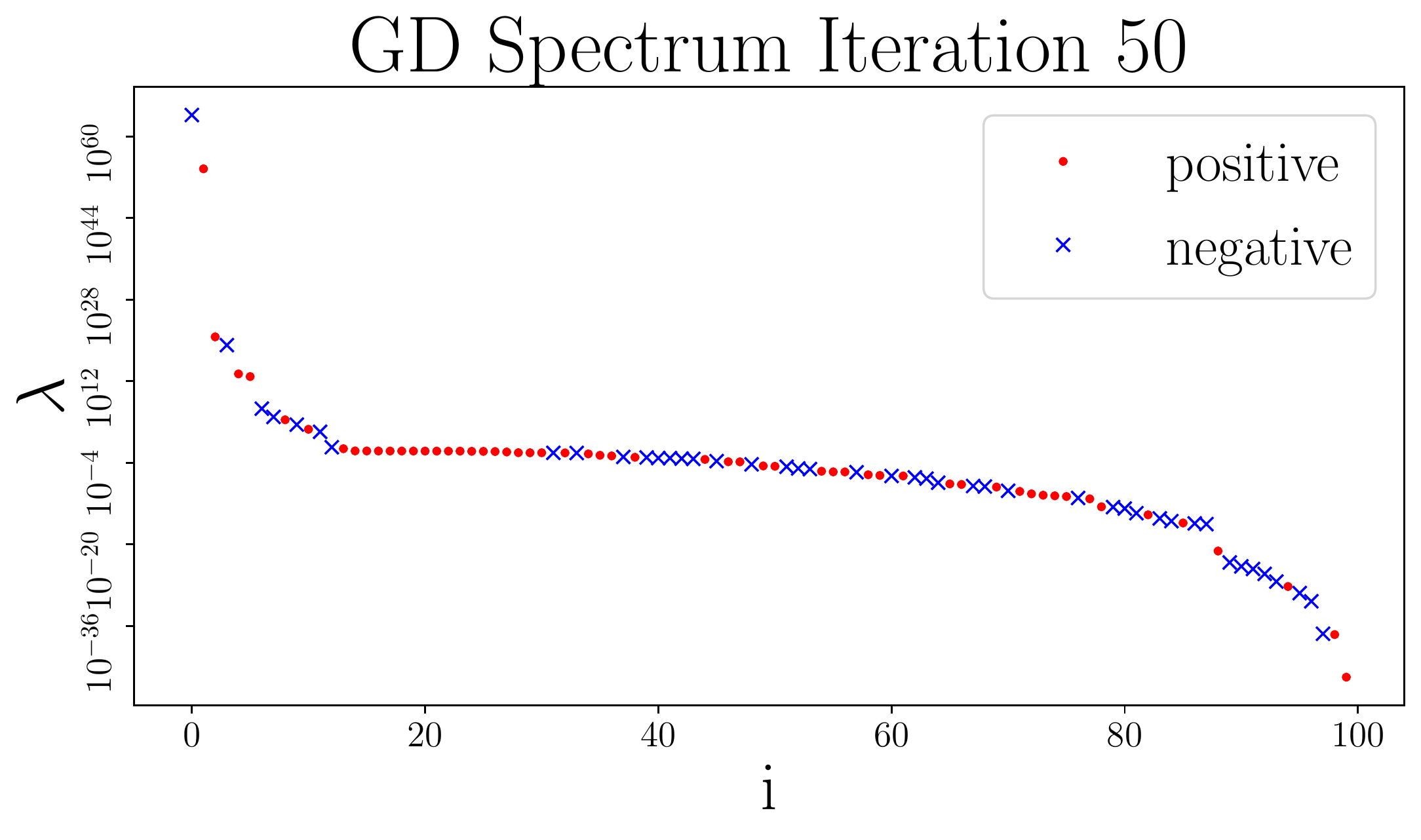}
\end{subfigure}%
\begin{subfigure}{0.33\textwidth}
\includegraphics[width = \textwidth]{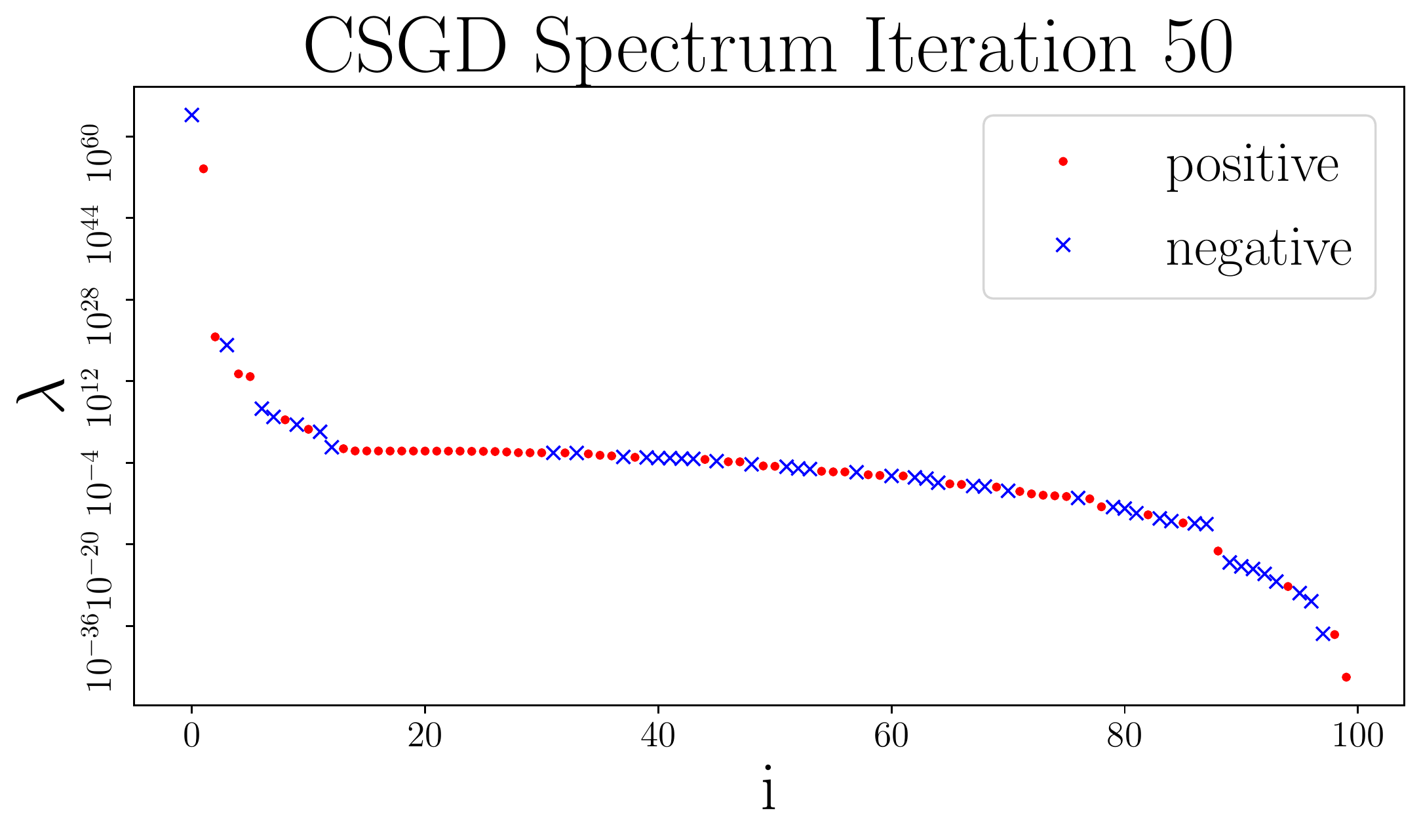}
\end{subfigure}
\begin{subfigure}{0.33\textwidth}
\includegraphics[width = \textwidth]{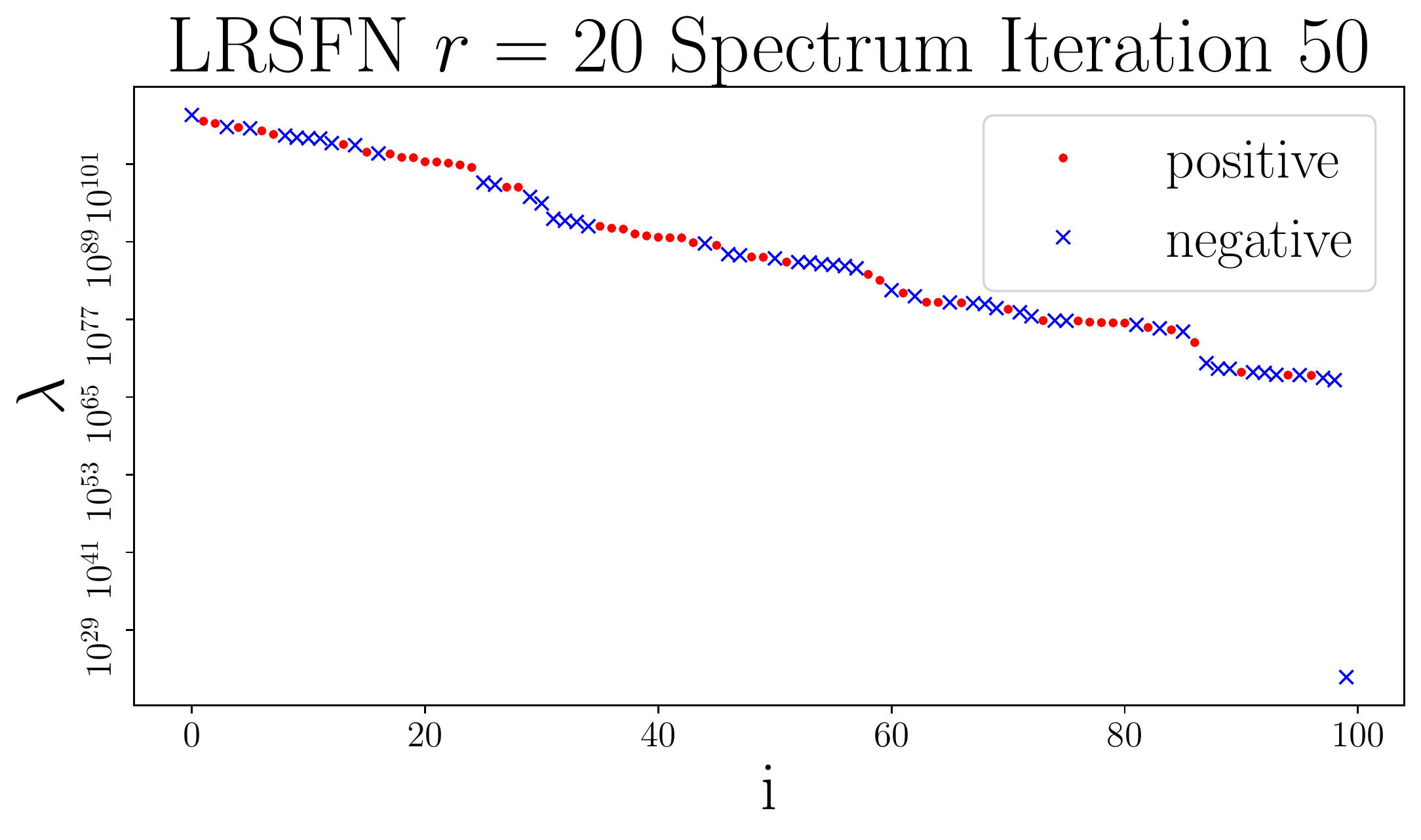}
\end{subfigure}%
\begin{subfigure}{0.33\textwidth}
\includegraphics[width = \textwidth]{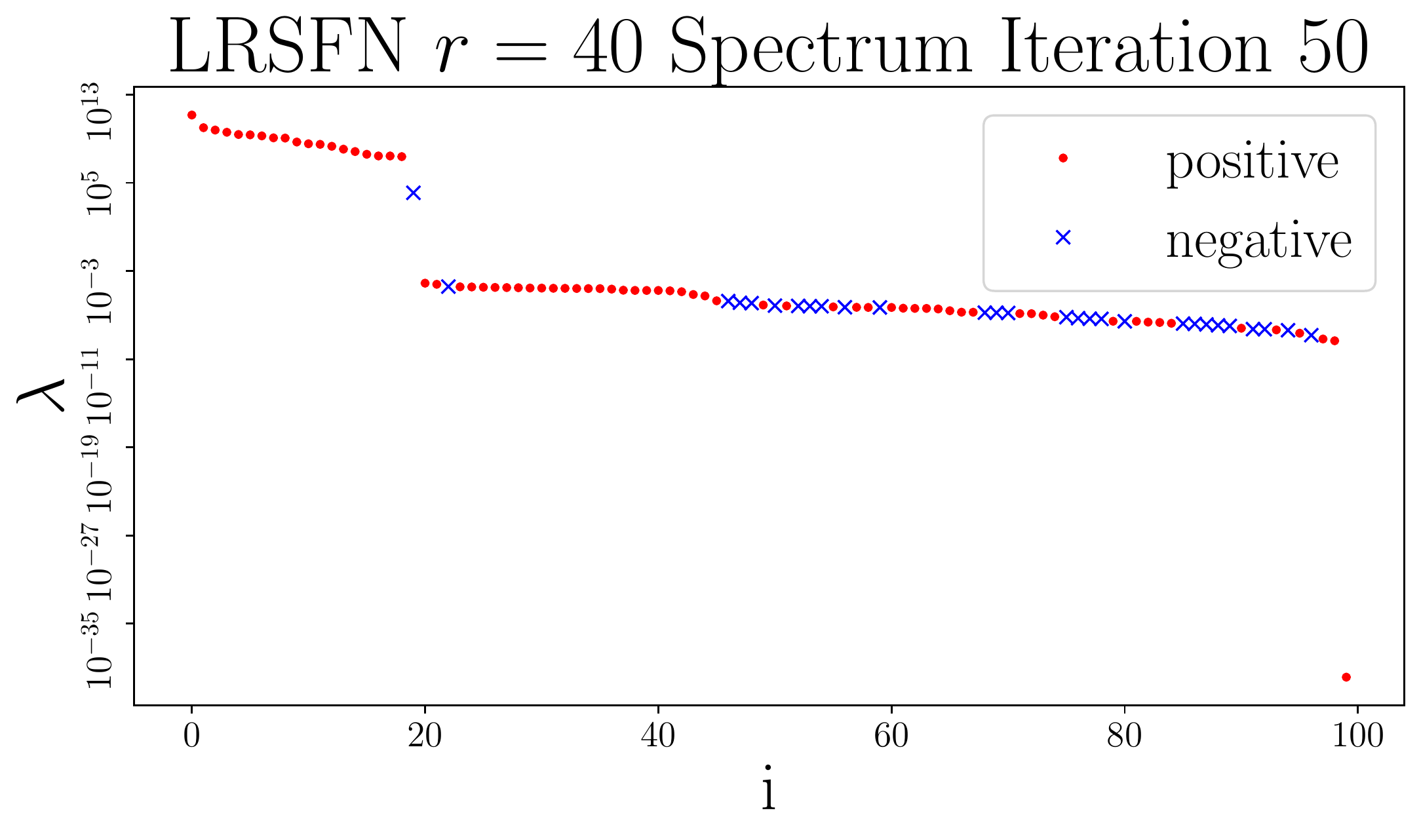}
\end{subfigure}%
\begin{subfigure}{0.33\textwidth}
\includegraphics[width = \textwidth]{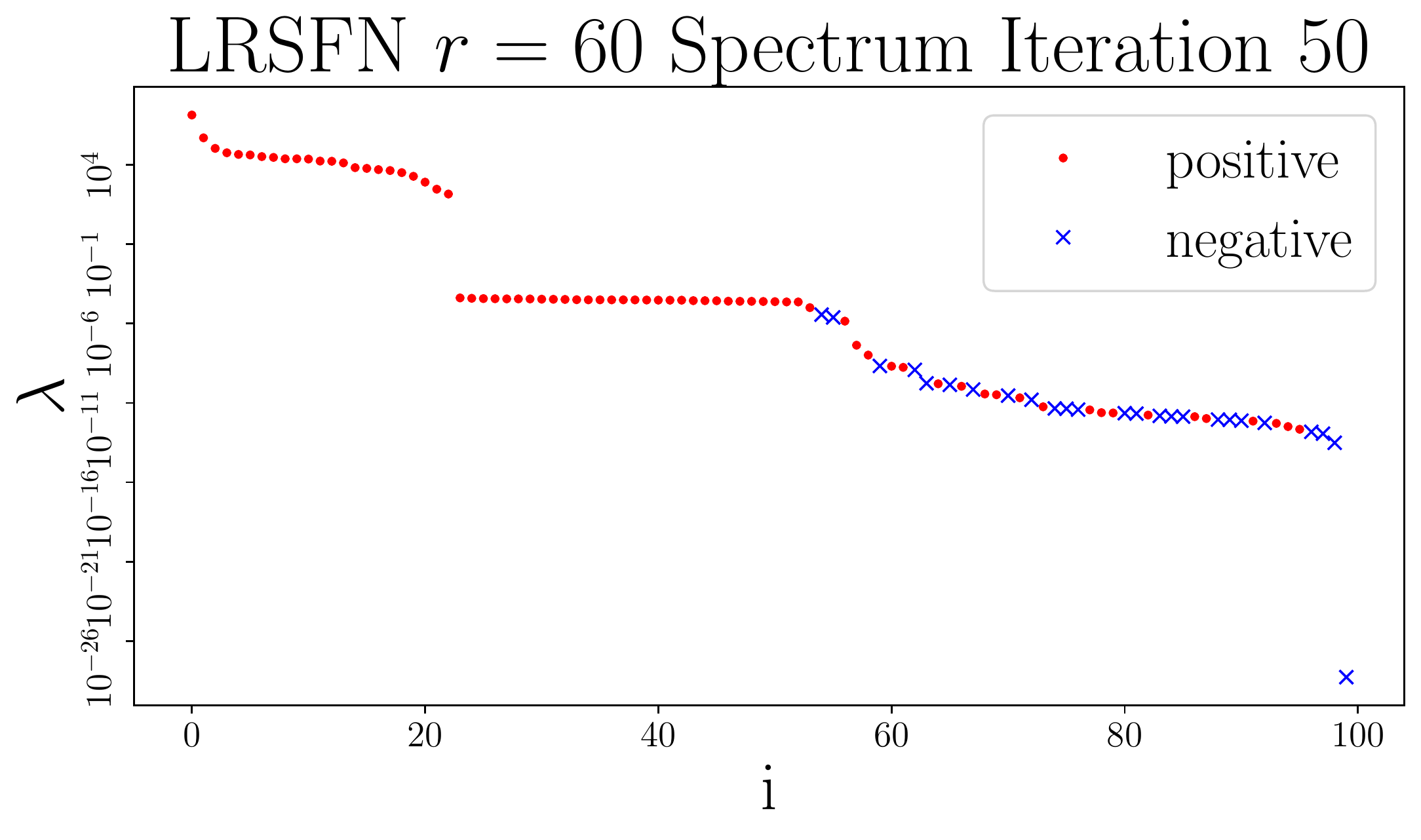}
\end{subfigure}
\caption{Various Hessian spectra for Michalewicz optimization}
\label{michalewicz_spectra}
\end{figure}

All of the optimizers tested for this problem were \emph{stable} in the sense that they did not diverge, but Newton and GD oscillated and performed poorly. The poor performance of Newton is due to the indefiniteness of the Hessian; the iterates oscillate since descent is not guaranteed. Gradient descent needs to be rescaled and curvature information is useful for escaping ``valleys'' for this problem. As seen in Figure \ref{michalewicz_losses}, CSGD performs much better than GD with $\alpha = 1$ for this problem: showing the difficulty of finding optimal scaling for GD (and other related first order methods) in practice. For this problem, full SFN performs the best of all the methods. The right plot in Figure \ref{michalewicz_losses}, and Table \ref{table_michalewicz} demonstrate that LRSFN essentially interpolates between GD and full SFN on rank. Spectral plots shown in Figure \ref{michalewicz_spectra} show the Hessian spectrum at initial guess, and at the 50th iteration of GD,CSGD and LRSFN for $r = 20,40,60$. All plots demonstrate the indefiniteness of the problem; the plots for GD and CSGD demonstrate the issues first order methods may have escaping indefiniteness. The comparison of LRSFN at iteration 50 show that for $r =40,60$ LRSFN is able to roughly resolve indefiniteness in the first $r$ modes. For the choice of $r = 20$, LRSFN is unable to resolve indefiniteness, possibly explaining its general poor performance as seen in Table \ref{table_michalewicz} and Figure \ref{michalewicz_losses} (note that  interestingly CSGD performs better than LRSFN $r = 20$). This example demonstrates the viability of LRSFN and full SFN for problems that exhibit indefiniteness, and substantial rank decay (steepness). The next example (Rosenbrock) demonstrates that rank truncation can create instabilities when the Hessian has heavy tails. The d-dimensional Rosenbrock function,
\begin{equation}
  F(w) = \sum_{j=1}^{d-1}100(w_{j+1} -w_j^2)^2 + (1-w_j)^2,
\end{equation}
has heavy tails (approximately constant spectrum!), it is minimized when $w = \mathds{1}^T$. This example demonstrates a problem for which both first order methods and LRSFN struggle, but for which the full Newton methods are exceptional. In this example we take $d = 10$, and $w_0 = \mathbf{0}\in \mathbb{R}^{10}$. The losses vs iterations are summarized below in Figure \ref{rosenbrock_plots}.

\begin{figure}[H] 
\begin{subfigure}{0.5\textwidth}
\includegraphics[width = \textwidth]{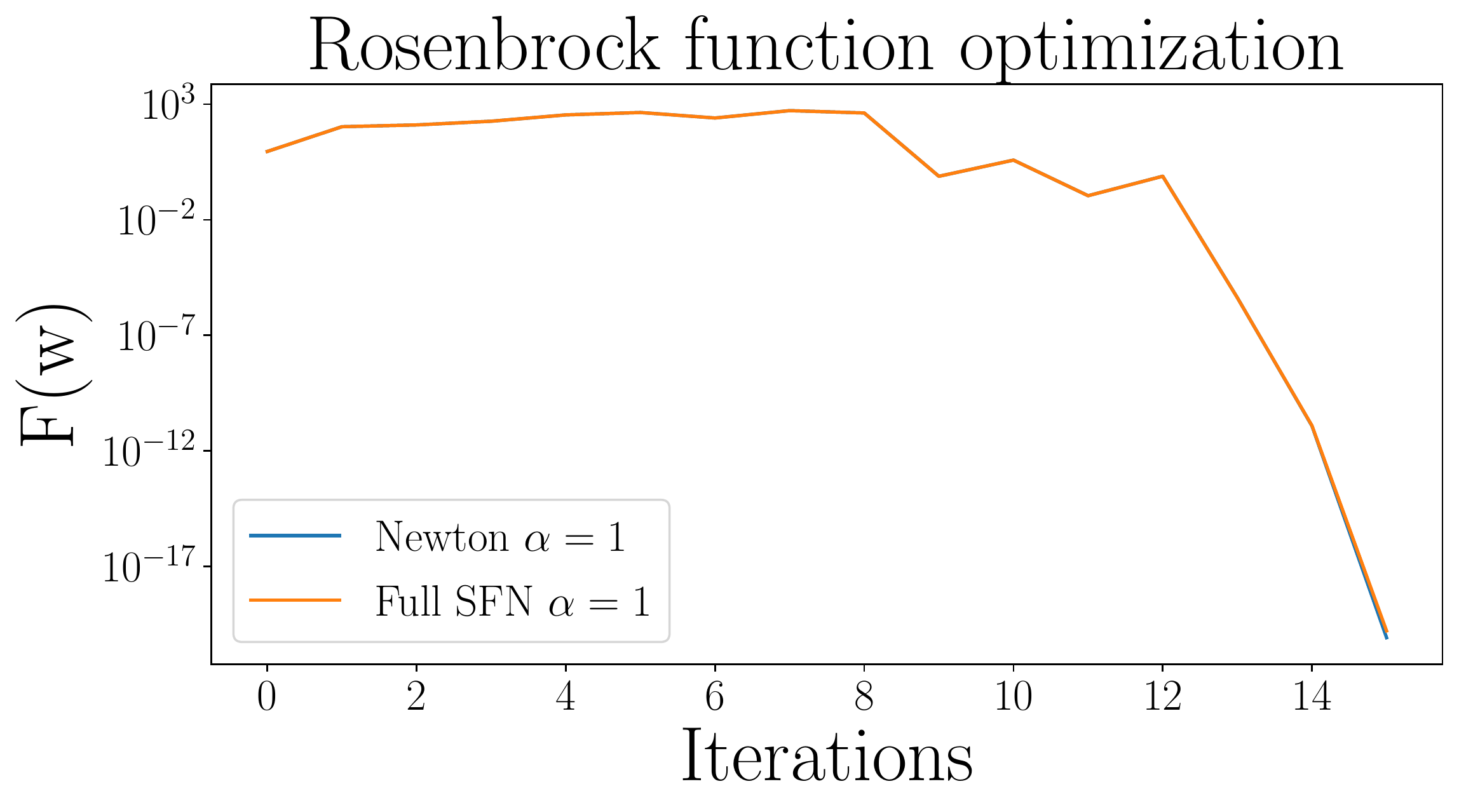}

\end{subfigure}%
\begin{subfigure}{0.5\textwidth}
\includegraphics[width = \textwidth]{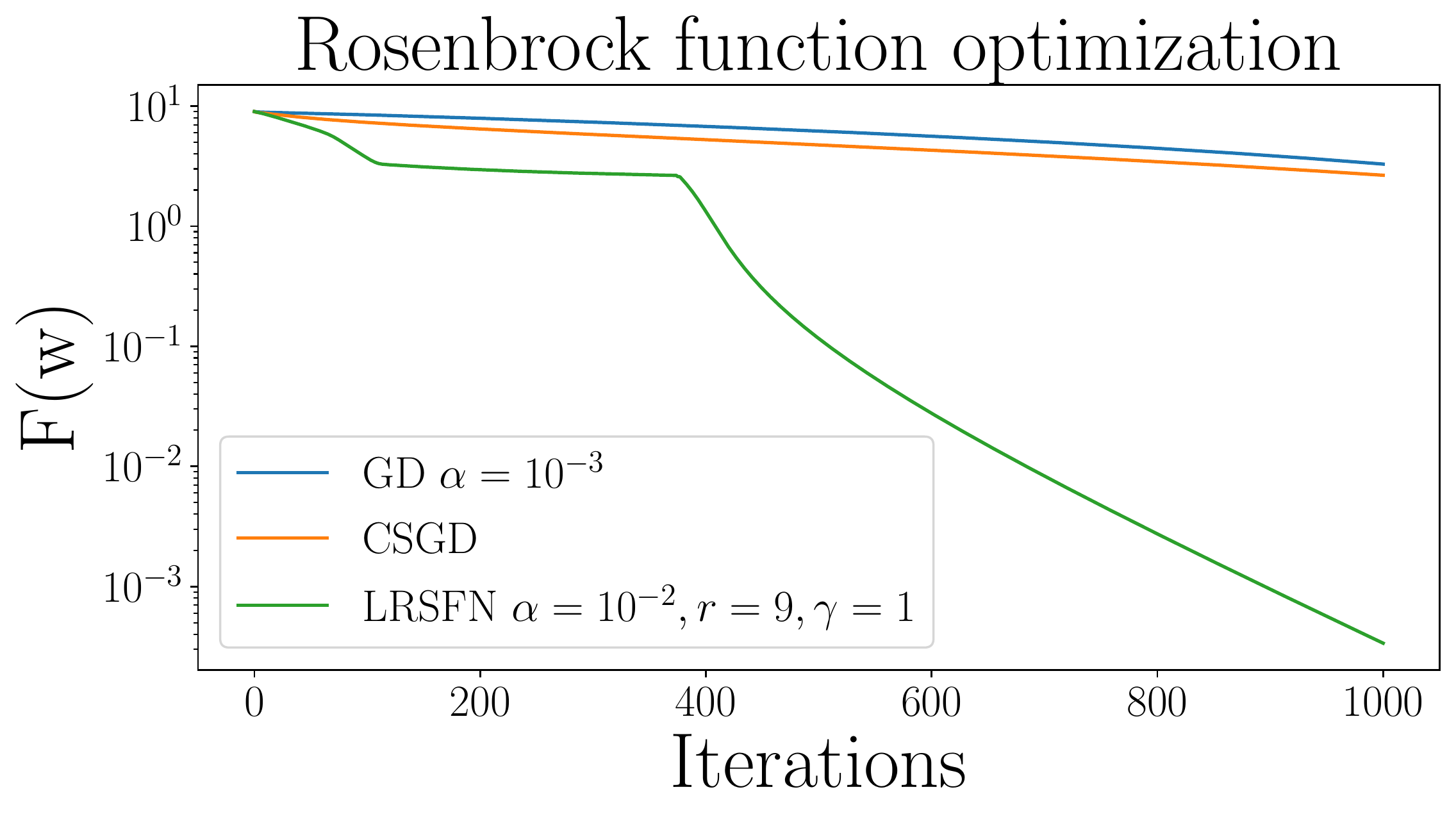}
\end{subfigure}
\caption{Comparison of methods for the 10 dimensional Rosenbrock function}
\label{rosenbrock_plots}
\end{figure}
Both Newton and full SFN converge to the the minimizer in 15 iterations for this problem, after first significantly increasing the objective function. GD is unstable for fixed learning rates $\alpha > 2\times 10^{-3}$, both GD $\alpha = 10^{-3}$ and CSGD are slow to converge. Interestingly LRSFN with $r = d-1 = 9$ performs very poorly for this problem. In order to improve stability and convergence, one has to take not only small steps, but also modify $\gamma$. The reason for this is because the objective function has heavy tails (specifically situations in which $|\lambda_d| $ is comparable to $|\lambda_1|$). LRSFN performs generally poorly for the d-dimensional Rosenbrock function, even with damping parameter $\gamma$ being increased in order to account for the heavy tail truncation. In the Michalewicz problem one can find reasonable $r$ such that the ratio of $|\lambda_1|/|\lambda_r|$ is a few orders of magnitude. For the Rosenbrock problem this is not the case, and this is a key issue for LRSFN; as one might expect, it is not well suited to problems with heavy tails.

In order to make use of LRSFN we argue three ingredients are necessary: very high dimension, $d$, a stochastic objective function that can be subsampled to reduce per-iteration cost and critically a spectrum that exhibits some substantial rank decay in early modes (note this is not necessarily ``low rank'', per se). The significant reduction in a small subspace means that we can get away with compressing a small portion of the spectrum, and the stochasticity means that we can further bring down the cost of the Hessian matrix-vector products relative to the gradient by subsampling the Hessian even further. The next example highlights the strenghts of LRSFN, because it possesses all of these features.

\subsection{CIFAR\{10,100\} transfer learning}

Transfer learning is a common deep learning paradigm used extensively in industrial and academic applications \cite{WeissKhoshgoftaarWang2016}. Pre-trained models can make use of extensive datasets representative of a particular application which can then be used as initial guesses for training in specific applications where few data are available. In the following two examples we use a ResNet50 neural network \cite{HeZhangRenEtAl2016}, pre-trained on ImageNet \cite{DengDongSocherEtAl2009}, which then is adapted to CIFAR\{10,100\} classification \cite{KrizhevskyNairHinton2010}. For the CIFAR10 example (which has 10 classes), the resulting network weight dimension is $d = 21,600,074$, while for CIFAR100 (which has 100 classes) the weight dimension is $d = 28,035,044$. For both datasets, all $50,000$ training data are used, of the $10,000$ testing data, $2,000$ are set aside for validation, and the other $8,000$ are used for testing post training. The keras default of $N_{X_k}=32$ are used for gradient subsampling. We take the Hessian subsampling to be $N_{S_k} = 8$ to demonstrate the opportunities for additional computational economy. The combination of these two sample sizes introduces a significant amount of Monte Carlo stochasticity in each iteration. For this reason, very small step lengths are required for LRSFN (as predicted by the analysis in Section \ref{section:stability}). We compare directly against Adam and SGD which are standard optimizers for this type of problem, we consider many different step lengths for each optimizer and report the best. The optimizers are evaluated based on best generalization accuracy for equivalent computation for neural network evaluation, which is the dominant computational cost. We term this equivalent compute metric ``epoch equivalent work'', for first order methods it corresponds to simply epochs, while for second order methods it takes into account the additional queries of the neural network to form the stochastic Hessian. For one iteration of a first order method $N_{X_k}$ evaluations of the neural network are required to compute the gradient, for one iteration of rank $r$ LRSFN, $N_{X_K} + 2rN_{S_k}$ evaluations are required to compute the gradient and Hessian. We use rank $r = 40$ to demonstrate that highly compressed Hessian spectra can be useful so long as the spectrum has some substantial decay. At the end of each equivalent epoch, validation accuracy is computed and the best weights found during training validation are returned from the optimization procedure.

The goal for this problem is to find the weights $w^*$ with the best generalization accuracy to unseen data, this goal is independent of what loss function is used. For all examples we test out two different loss functions: categorical cross entropy, and a sum of categorical cross entropy and least-squares misfit. The latter was found to work well for LRSFN by an accident. We report the best for each optimizer across both loss functions. Results did not change drastically between the two loss functions.

Note that we do not compare wall-clock directly since our code is an unoptimized research code, and the Adam and SGD are highly optimized industry codes in tensorflow/keras, where automatic differentiation is explicitly optimized for first order methods, so it would not constitute a meaningful comparison. Details about the implementation can be found here \cite{hessianlearn}. Results are reported marginalized over 5 different seeds for the weights and partitioning of the training data. For one of the CIFAR10 seeds, LRSFN blew up at the initial guess, so this seed (2) was skipped. Instabilities (poor conditioning) at initial guesses are common for Newton methods in other settings, in these cases often a first order method or quasi-Newton method is employed for a fixed number of iterations first before switching to Newton to avoid these issues (this is done in hIPPYlib \cite{VillaPetraGhattas20}). Such a routine is outside of the scope of this work, where we want to compare optimizers directly, so we simply skipped this one seed, so for CIFAR10 the results report averages over seeds for which LRSFN did not blow up at the initial guess; no such issue occurred for CIFAR100. For reproducibility the seeds used for CIFAR10 were $\{0,1,3,4,5\}$, the seeds used for CIFAR100 were $\{0,1,2,3,4\}$. Due to the many convolution layers training for these types of models typically require access to modern GPUs, where the convolution operations can be executed quickly. These results were run on a machine with two NVIDIA A100 GPUS, each with 40GB of RAM, with a CUDA enabled installation of tensorflow 2.4.1. Only one GPU was used for each optimization. 

The best step lengths for Adam ($\alpha = 10^{-4}$) and SGD ($\alpha = 10^{-1}$) were the same for both CIFAR10 and CIFAR100. LRSFN was unstable for step lengths that we tested greater than $\alpha = 10^{-4}$, and performed the best for $\alpha = 10^{-5}$. The best validation and testing accuracies are shown below in Table \ref{table_cifar}, plots of validation vs accuracy for those runs are shown in Figure \ref{validation_vs_accuracy_cifar}.

\begin{table}[H] 
\center
\begin{tabular}{|l|c|c|c|c|}
\hline
\enskip & \multicolumn{2}{c|}{CIFAR10} & \multicolumn{2}{c|}{CIFAR100} \\
\hline
\enskip & Validation & Testing & Validation & Testing \\
\hline
SGD  & $91.9 \pm 0.55\% $ & $91.5 \pm 0.50\%$ & $70.0 \pm 0.77\% $ & $69.2 \pm 0.18\%$\\ 
\hline
Adam  & $92.3 \pm 0.23\% $ & $91.5 \pm 0.52\%$ & $69.5 \pm 1.11\% $ & $69.0 \pm 0.32\%$\\ 
\hline
LRSFN  & $\mathbf{94.5 \pm 0.27\% }$& $\mathbf{93.8 \pm 0.09\% }$ & $\mathbf{76.0 \pm 0.48\%}$& $\mathbf{75.7 \pm 0.32\%}$\\ 
\hline
\end{tabular}
\caption{Accuracy results for CIFAR\{10,100\} transfer learning (mean $\pm$ one standard deviation). For SGD $\alpha = 10^{-1}$, for Adam $\alpha = 10^{-4}$, for LRSFN $\alpha = 10^{-5}$ and $r = 40$, all fixed.}
\label{table_cifar}
\end{table}


\begin{figure}[H] 
\begin{subfigure}{0.5\textwidth}
\includegraphics[width = \textwidth]{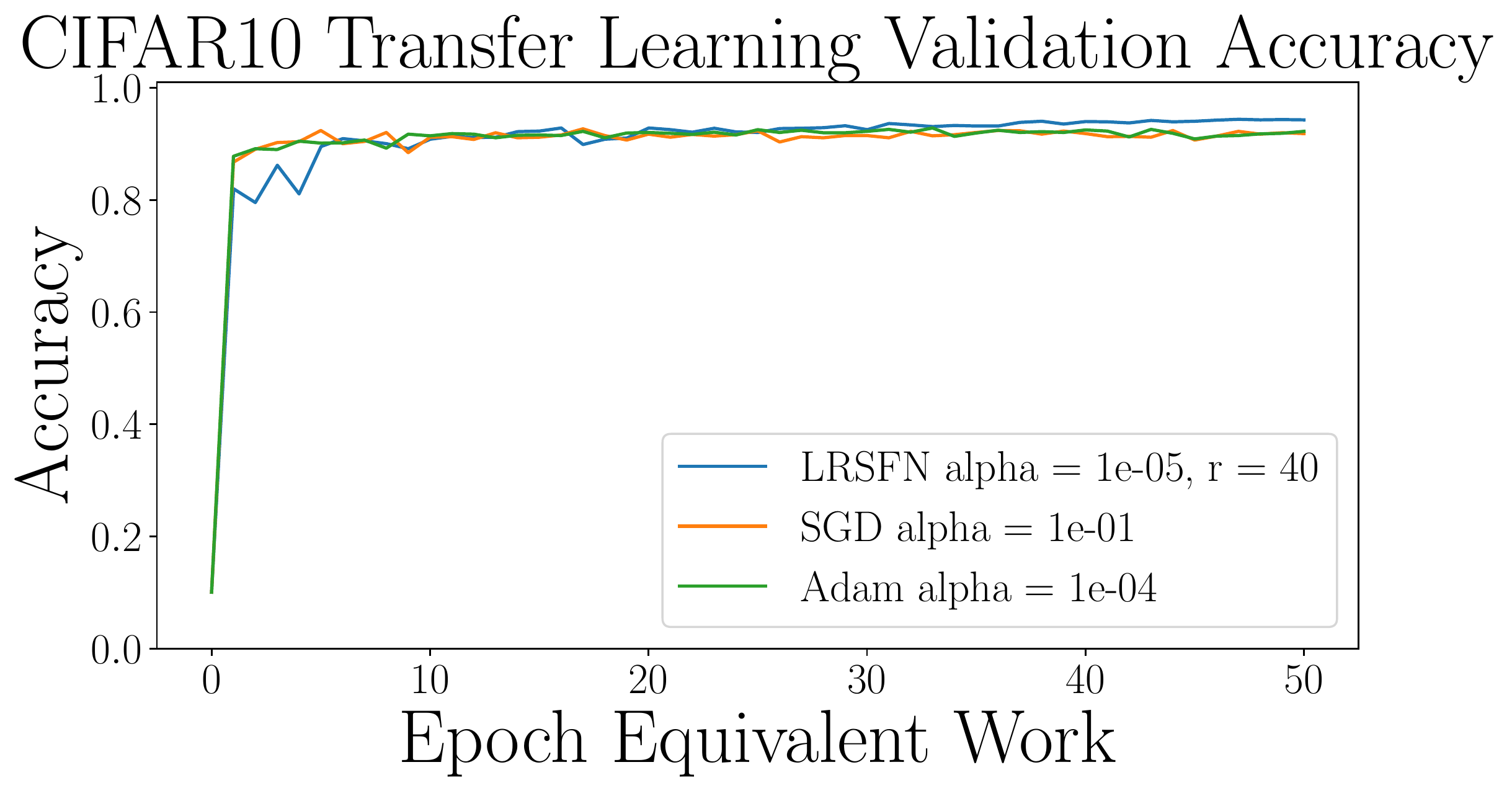}

\end{subfigure}%
\begin{subfigure}{0.5\textwidth}
\includegraphics[width = \textwidth]{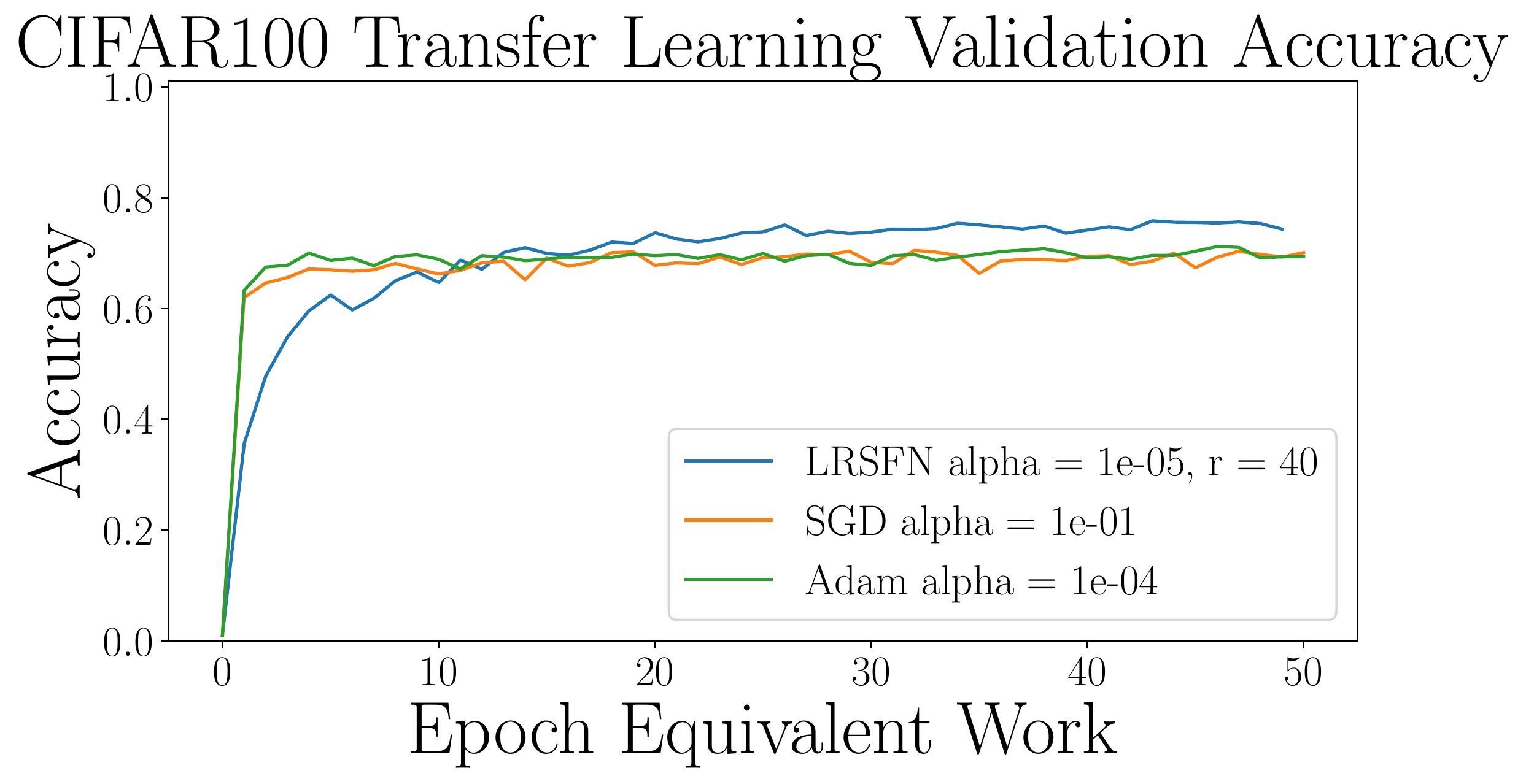}
\end{subfigure}
\caption{Validation accuracy vs training for CIFAR\{10,100\} Transfer Learning}
\label{validation_vs_accuracy_cifar}
\end{figure}
In 50 compute equivalent epochs LRSFN was able to find weights that generalized better for CIFAR10, and significantly better for the much more difficult CIFAR100 problem. In both cases LRSFN doesn't overtake SGD and Adam until around the tenth epoch or so. The spectral decay (in the first 80 modes) are shown below for CIFAR100 training in Figure \ref{cifar100_spectra}, CIFAR10 had quicker spectral decay.

\begin{figure}[H] 
\begin{subfigure}{0.5\textwidth}
\includegraphics[width = \textwidth]{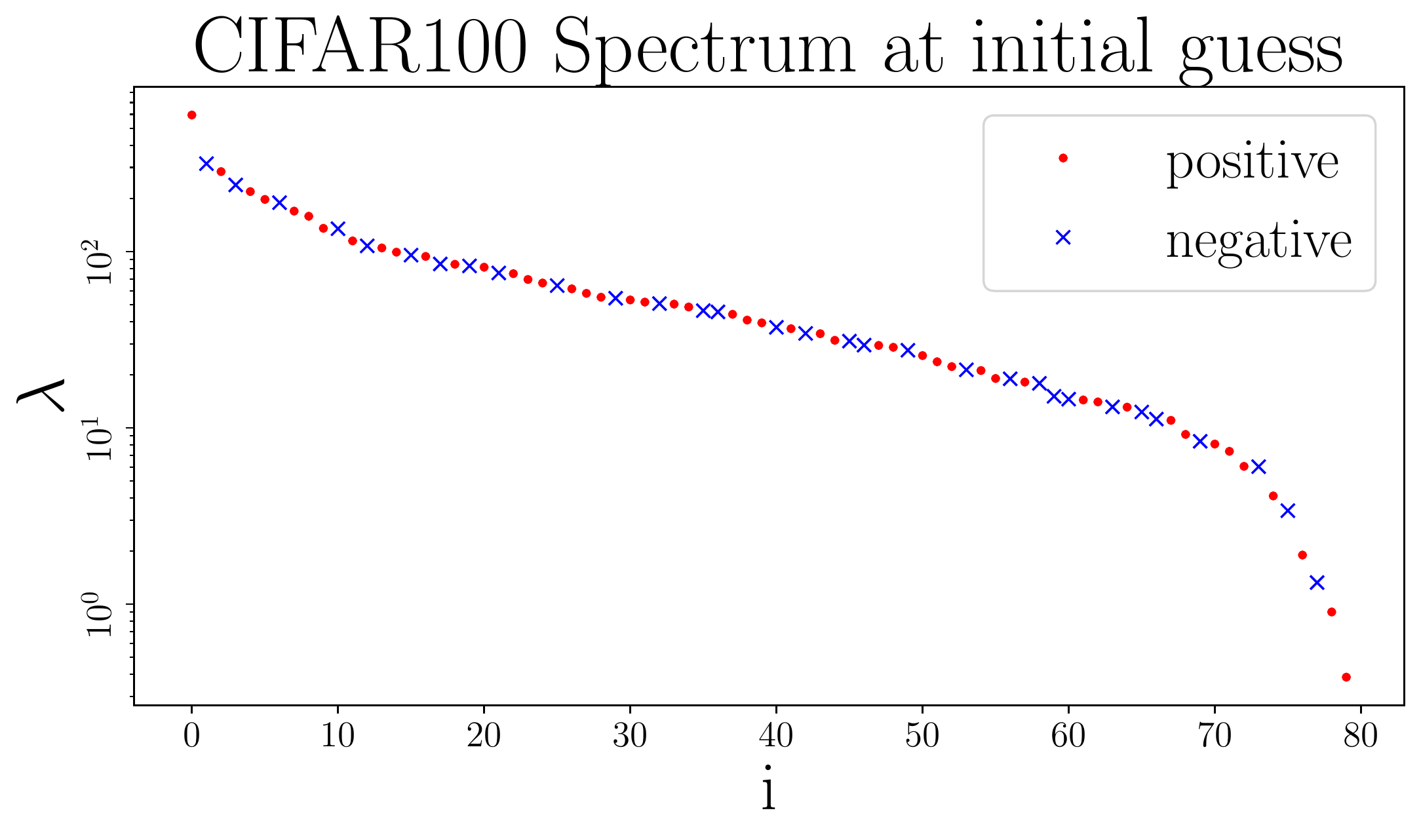}
\end{subfigure}%
\begin{subfigure}{0.5\textwidth}
\includegraphics[width = \textwidth]{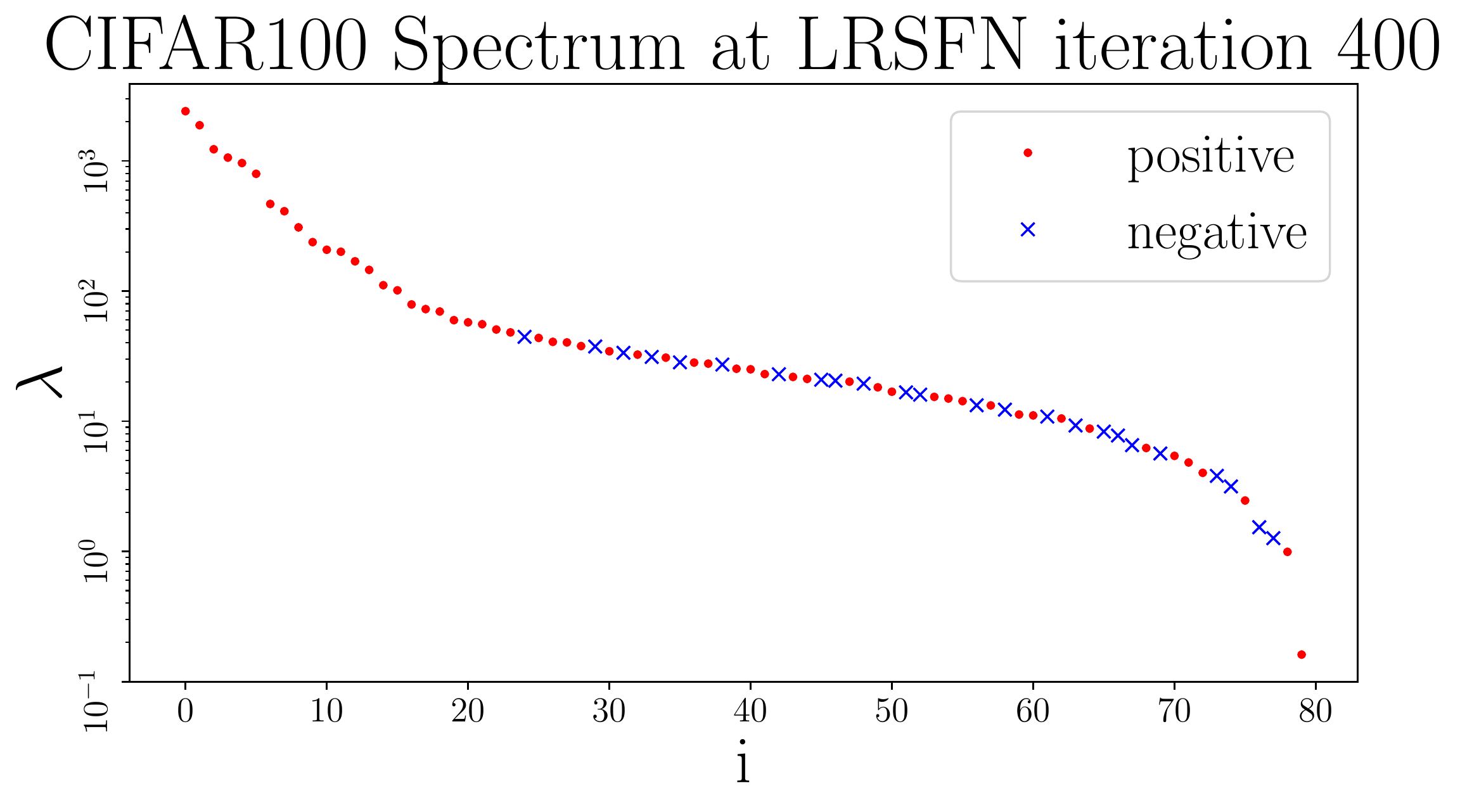}
\end{subfigure}
\caption{CIFAR100 Spectra at initial guess, and later during LRSFN training}
\label{cifar100_spectra}
\end{figure}

This problem has numerical rank significantly higher than $40$, but the spectrum is reduced 2-3 orders of magnitude in the first $40$ modes, suggesting that LRSFN can work on problems that do not necessarily have low rank, so long as they have some significant spectral decay in the first $r$ modes. For many problems picking $r$ to be the actual numerical rank for the problem will make the per-iteration computational cost far too large to be of practical use. The weights are highly indefinite at the initial guess, but as training progresses the dominant eigenvalues start to become more and more positive, which is what one should hope for in this case: to find a candidate solution that has positive Hessian spectrum, this numerical finding is consistent with findings in \cite{OLearyRoseberryAlgerGhattas2019}. This numerical result demonstrates that second order stochastic methods can be of use in modern deep learning problems, many of which have the key ingredients for LRSFN: large parameter dimension, highly stochastic objective function and rank decay.

\section{Conclusion}

This work makes the case for LRSFN as a second order optimizer in the stochastic approximation (SA) setting. By incorporating specific information about the indefiniteness of the optimization landscape, LRSFN can facilitate fast escape from indefinite regions and converge to high quality solutions. Additionally, LRSFN is particularly well suited to modern parallel computing environments, since it leverages parallelism in the approximation of the Hessian via robust randomized methods. Second order methods are heavily studied and used in deterministic  and sample average approximation (SAA) optimization problems. They are not used often in the SA setting, since convergence properties are hobbled by the Monte Carlo errors that are introduced via sampling the Hessian and the gradient \cite{BollapragadaByrdNocedal2018,OLearyRoseberryAlgerGhattas2019}. However, as our numerical results demonstrate, a highly subsampled LRSFN method can outperform popular first order methods on large scale deep learning tasks in terms of generalizability for equivalent computational work.

In the SA setting, where stochasticity dominates the iterative update procedure, globalization procedures are not useful since the statistical estimators of descent are likely to be unreliable. In this setting, one must take small steps, and step lengths should take into account the features of the optimization landscape as well as the level of stochastic noise in the updates. In Section \ref{section:stability}, we investigate the stability of the LRSFN update by first conceiving of the stochastic optimization procedure as simulation of expected risk minimizing flow via explicit Euler. Each update introduces deviations from the minimizing trajectory due to stochastic error. We derive step length conditions so that these deviations do not pose major instabilities. In this analysis we are able to derive bounds for the stability of gradient descent, Newton, and LRSFN. For gradient descent (and first order optimizers in general), these bounds demonstrate that the effects of optimization geometry (i.e. Lipschitz constants, curvature) can be fully decoupled from the effects of gradient Monte Carlo error. In this case the simulated dynamics are stable as long as the step taken satisfies both a condition for the geometry and a condition for the stochasticity (see Theorem \ref{theorem_stability_grad_descent}).

For second order stochastic optimizers, similar bounds are derived, but a key difference exists: the Monte Carlo error cannot be decoupled from the optimization geometry, since the Monte Carlo error involves the product of a stochastic matrix with a stochastic gradient. In this case, ill-conditioning of the Hessian may significantly amplify statistical errors in the gradient, and a large gradient norm can amplify the effects of the Hessian stochastic error (see Proposition \ref{proposition_newton_lrsfn_search}). These results demonstrate that even though one can take larger steps with second order methods reliably for deterministic problems, the converse may be the case for highly stochastic problems: one may have to take very small steps with second order stochastic optimizers.

Numerical results demonstrate that LRSFN can be of great use for highly stochastic high dimensional problems. The last two numerical results show that LRSFN outperform Adam and SGD on high dimensional transfer learning applications, CIFAR\{10,100\} with ResNet50; in the case of CIFAR100 LRSFN drastically outperformed Adam and SGD. We expect that LRSFN may be of use for extremely high dimensional difficult problems in machine learning such as classification transfer learning. In these cases, perhaps the use of Hessian curvature information brings in additional information into the update procedure that facilitates fast escape from indefinite regions and helps guard against overfitting. This would be consistent with what we have seen in the CIFAR\{10,100\} examples.

LRSFN has had difficulty in other problems with heavy tails, such as variational autoencoder training. More work needs to be done to investigate effective strategies to tackle Hessian approximations when the Hessian cannot be compressed efficiently by randomized eigenvalue decompositions. Another issue that can be taken into account to improve the performance of second order optimizers is dealing with ill-conditioning at initial guesses. Sometimes the spectrum is highly indefinite, and has very heavy tails at initial guesses but then the spectrum drops off at subsequent iterations. Occasionally initial steps taken from an initial guess can set second order methods off on a bad trajectory. One obvious workaround is to use a first order or quasi-second order method for a few iterations initially to move away from regions that are difficult for stochastic second order methods. This is done in many other settings such as inverse problems. When switching over to the stochastic second order method, larger steps may be able to be taken, which can lead to faster convergence. 

For the deep transfer learning problem, we do not compare methods with respect to wall-clock time, since we are comparing our research code against highly optimized industry codes where automatic differentiation is optimized specifically for first order methods. We instead compare the methods in terms of equivalent neural network computations, showing the ability of LRSFN to generalize better than first order methods for the same amount of neural network evaluations. Since Hessian-vector products resemble gradient computations, matrix free Newton methods, such as LRSFN, have the ability to take advantage of all the vectorization, compiler and automatic differentiation optimization, and domain specific hardware that first order methods have enjoyed, provided comparable attention to implementation and tuning is paid.

\section*{Acknowledgements}
Conversations with Anna Yesypenko, Alex Ames, Brendan Keith and Rachel Ward were helpful during the preparation of this manuscript.



\bibliographystyle{siamplain}
\bibliography{local}

\appendix

\section{Additional Analysis for Stability Bounds} \label{section:appendix_bounds}

\begin{lemma}{Monte Carlo Error for Gradients}
\label{grad_mc_lemma}
If Assumption \ref{assumption_grad_mc} is met, then the gradient Monte Carlo error is bounded as follows:
\begin{equation}
    \mathbb{E}_k[\|\nabla F(w_k) - \nabla F_{X_k}(w_k)\|] \leq \frac{v}{\sqrt{N_{S_k}}}
\end{equation}
\end{lemma}

\begin{proof}
This is a standard result, see for example Theorem 2.1 and Lemma 2.2 in \cite{BollapragadaByrdNocedal2018}. We recount it here. By Jensen's inequality and the assumption we have
\begin{align}
\left( \mathbb{E}_k[\|\nabla F(w_k) - \nabla F_{X_k}(w_k)\|]\right)^2 &\leq \left( \mathbb{E}_k[\|\nabla F(w_k) - \nabla F_{X_k}(w_k)\|^2]\right) \nonumber \\ 
&= \text{tr}\left( \text{Cov}\left(\frac{1}{N_{X_k}}\sum_{i=1}^{N_{X_k}}\nabla F_i(w_k) \right) \right) \nonumber \\
&\leq \frac{1}{N_{X_k}}\text{tr}\left( \text{Cov}\left(\nabla F_i(w_k) \right) \right) \leq \frac{v^2}{N_{X_k}}.
\end{align}
\end{proof}

\begin{lemma}{Monte Carlo Error for Hessians}
\label{hessian_mc_lemma}
If Assumption \ref{assumption_hess_mc} holds, then the Monte Carlo error for the subsampled Hessian is as follows
\begin{equation}
    \mathbb{E}_k[\| \nabla^2 F_{S_k}(w_k) - \nabla^2 F(w_k) \|_{\ell^2(\mathbb{R}^{d\times d})}] \leq \frac{\sigma}{\sqrt{N_{S_k}}},
\end{equation}
and for the subsampled absolute Hessian:
\begin{equation}
    \mathbb{E}_k[\| |\nabla^2 F_{S_k}(w_k)| - |\nabla^2 F(w_k)| \|_{\ell^2(\mathbb{R}^{d\times d})}] \leq \frac{\sigma^\text{abs}}{\sqrt{N_{S_k}}}
\end{equation}
\end{lemma}
This is another standard result. The proof is almost identical to Lemma \ref{grad_mc_lemma}, so we omit it for brevity. See Lemma 2.3 in \cite{BollapragadaByrdNocedal2018} for the details of a similar result.

We prove generic results for Monte Carlo errors in search directions of the form $p_k = A^{-1}g$ (and subsampled versions $p_k = A_{S_k}^{-1}g_{X_k}$, where $A$ and $g$ are unspecified, but satisfy Monte Carlo bounds.

The next lemma establishes Monte Carlo errors for the inverse of a matrix ($A^{-1}$), given that the matrix itself has a bounded Monte Carlo error. It is a conservative bound, but shows how the Monte Carlo errors of the stochastic inverse can be sensitive to the conditioning of the matrix being inverted.
\begin{lemma}{Monte Carlo Error for Stochastic Matrix Inverse}
\label{stochastic_inverse_mc_lemma}

Let $A_{S_k}$ be a stochastic matrix with bounded Monte Carlo error $E_\text{MC}^A = A_{S_k} - A$,
\begin{equation}
  \mathbb{E}_k[\|E_\text{MC}^A\|_2] \leq \frac{\sigma^A}{\sqrt{N_{S_k}}}, 
\end{equation}
further, assume $A_{S_k}^{-1}$, $A^{-1}$,$(I + E_\text{MC}^AA^{-1})^{-1}$ and $(I + \frac{1}{2}E_\text{MC}^AA^{-1})^{-1}$ exist, then
\begin{equation}
  \mathbb{E}_k[\|A_{S_k}^{-1} - A^{-1}\|] \leq \frac{C_1}{\sqrt{N_{S_k}}} + \frac{C_2}{N_{S_k}},
\end{equation}
with 
\begin{align}
C_1 &= \frac{\sigma^A}{2}\|A^{-1}\|^2  (1 + \mathbb{E}_k[\|(I + E_\text{MC}^AA^{-1})^{-1}\|]) \\
C_2 &= \frac{(\sigma^A)^2}{4}\|A^{-1}\|^3  \mathbb{E}_k\left[\left\|\left(I +\frac{1}{2}E_\text{MC}^AA^{-1}\right)^{-1}\right\|\right]
\end{align}

\end{lemma}

\begin{proof}
We have the following equality
\begin{equation}
\label{three_term_generic_inverse_mc}
  A_{S_k}^{-1} - A^{-1} = (A_{S_k} + A)^{-1}(A_{S_k} + A)(A_{S_k}^{-1} - A^{-1}).
\end{equation}
The conditional expectation with respect to $S_k$ of the norm of the first term can be bounded using the Woodbury formula, triangle inequality and Cauchy-Schwarz inequality:
\begin{subequations}
\begin{align}
    (A_{S_k} + A)^{-1} = (2A + E_\text{MC}^A)^{-1} = \frac{1}{2}A^{-1} - \frac{1}{4}A^{-1}(I + E_\text{MC}^AA^{-1})^{-1}E_\text{MC}^A A^{-1}\\
    \mathbb{E}_k[\|(A_{S_k} + A)^{-1}\|] \leq \frac{1}{2}\|A^{-1}\| + \frac{1}{4}\|A^{-1}\|^2 \frac{\sigma^A}{\sqrt{N_{S_k}}}\mathbb{E}_k\left[\left\|\left(I + \frac{1}{2}E_\text{MC}^AA^{-1}\right)^{-1}\right\|\right]
\end{align}
\end{subequations}
Multiplying out the last two terms in \eqref{three_term_generic_inverse_mc} we have
\begin{subequations}
\begin{equation}
    (A_{S_k} + A)(A_{S_k}^{-1} - A^{-1}) = I - A_{S_k}A^{-1} +AA_{S_k}^{-1} - I = AA_{S_k}^{-1} - A_{S_k}A^{-1}.
\end{equation}
Decomposing the stochastic approximation $A_{S_k} = A + E_\text{MC}^A$, and employing the Woodbury formula we get the following bounds:

\begin{align}
    &AA_{S_k}^{-1} = A(A+E_\text{MC}^A)^{-1} = A(A^{-1} - A^{-1}(I + E_\text{MC}^AA^{-1})^{-1}E_\text{MC}A^{-1}) \nonumber \\
    &          \qquad\qquad\qquad\qquad\qquad\qquad                          = I -(I+E_\text{MC}^AA^{-1})^{-1}E_\text{MC}A^{-1} \\
    &A_{S_k}A^{-1} = (A+E_\text{MC}^A)A^{-1} =   I + E_\text{MC}^AA^{-1}              \\
    &AA_{S_k}^{-1} - A_{S_k}A^{-1} = -(I + (I+E_\text{MC}^AA^{-1})^{-1})E_\text{MC}A^{-1}    .                
\end{align}
Which leads to the following bound of the conditional expectation of the norm:
\begin{align}
    \mathbb{E}_k[\|(A_{S_k} + A)(A_{S_k}^{-1} - A^{-1}) \|] &\leq \|A^{-1}\|\frac{\sigma^A}{\sqrt{N_{S_k}}}\mathbb{E}_k[\|I + (I+E_\text{MC}^AA^{-1})^{-1}\|] \\
        &\leq \|A^{-1}\|\frac{\sigma^A}{\sqrt{N_{S_k}}} \left(1 + \mathbb{E}_k[\|I+E_\text{MC}^AA^{-1}\|] \right)
\end{align}

\end{subequations}

By the Cauchy-Schwarz inequality we can bound the conditional expectation of the norm of \eqref{three_term_generic_inverse_mc}
\begin{subequations}
\begin{equation}
  \mathbb{E}_k[\|A_{S_k}^{-1} - A^{-1}\|] \leq \frac{C_1}{\sqrt{N_{S_k}}} + \frac{C_2}{N_{S_k}},
\end{equation}
with
\begin{align}
C_1 &= \frac{\sigma^A}{2}\|A^{-1}\|^2  (1 + \mathbb{E}_k[\|(I + E_\text{MC}^AA^{-1})^{-1}]) \\
C_2 &= \frac{(\sigma^A)^2}{4}\|A^{-1}\|^3  \mathbb{E}_k\left[\left\|\left(I +\frac{1}{2}E_\text{MC}^AA^{-1}\right)^{-1}\right\|\right].
\end{align}
\end{subequations}
\end{proof}

This allows us to trivially bound the Monte Carlo errors for stochastic Newton-like updates $p_k = A^{-1}_{S_k}g_{X_k}$.

\begin{lemma}{Monte Carlo Error for Stochastic Newton-like search direction}
\label{stochastic_newton_mc_lemma}

With the same assumptions as Lemma \ref{stochastic_inverse_mc_lemma}, additionally assume that there exists $v^g$, such that the stochastic right hand side (gradient-like term), $g_{X_k}$ has bounded Monte Carlo error: 
\begin{equation}
    \mathbb{E}_k[\|g_{X_k} - g\|] \leq \frac{v^g}{\sqrt{N_{S_k}}},
\end{equation}
the Monte Carlo error for a stochastic Newton-like update $p_k = A^{-1}_{S_k}g_{X_k}$ is bounded as follows:
\begin{subequations}
\begin{equation}
    \mathbb{E}_k[\|A^{-1}_{S_k}g_{X_k} - A^{-1}g\|] \leq \frac{C_0}{\sqrt{N_{X_k}}} + \frac{C_1}{\sqrt{N_{S_k}}}\mathbb{E}_k[\|g_{X_k}\|] +\frac{C_2}{N_{S_k}}\mathbb{E}_k[\|g_{X_k}\|],
\end{equation}
with
\begin{align}
C_0 &= v^g\|A^{-1}\| \\
C_1 &= \frac{\sigma^A}{2}\|A^{-1}\|^2  (1 + \mathbb{E}_k[\|(I + E_\text{MC}^AA^{-1})^{-1}]) \\
C_2 &= \frac{(\sigma^A)^2}{4}\|A^{-1}\|^3  \mathbb{E}_k\left[\left\|\left(I +\frac{1}{2}E_\text{MC}^AA^{-1}\right)^{-1}\right\|\right].
\end{align}
\end{subequations}

\end{lemma}
\begin{proof}
We can decompose the search direction as follows
\begin{subequations}
\begin{equation}
    A_{S_k}^{-1}g_{X_k} - A^{-1}g = A^{-1}g_{X_k} - A^{-1}g + A_{S_k}^{-1}g_{X_k} - A^{-1}g_{X_k},
\end{equation}
which leads to the following bound via the triangle inequality and Cauchy-Schwarz:
\begin{align}
\mathbb{E}_k[\|A_{S_k}^{-1}g_{X_k} - A^{-1}g \|] &\leq \mathbb{E}_k[\|A^{-1}(g_{X_k} - g)\|] + \mathbb{E}_k[\|(A_{S_k}^{-1} - A^{-1})g_{X_k}\|] \\
    &\leq \|A^{-1}\|\frac{v^g}{\sqrt{N_{X_k}}} + \mathbb{E}_k[\|A_{S_k}^{-1} - A^{-1}\|]\mathbb{E}_k[\|g_{X_k}\|].
\end{align}
\end{subequations}
The result follows from Lemma \ref{stochastic_inverse_mc_lemma}.
\end{proof}

\end{document}